\def\versionno{ FRDF     -- version  4.1 -- by jf --  29.1.09   }
\catcode`\@=11
\newif\if@fewtab\@fewtabtrue
{\count255=\time\divide\count255 by 60
\xdef\hourmin{\number\count255}
\multiply\count255 by-60\advance\count255 by\time
\xdef\hourmin{\hourmin:\ifnum\count255<10 0\fi\the\count255}}
\def\ps@draft{\let\@mkboth\@gobbletwo
    \def\@oddfoot{\hbox to 7 cm{\tiny \versionno
       \hfil}\hskip -7cm\hfil\rm\thepage \hfil {\tiny\draftdate}}
    \def\@oddhead{}
    \def\@evenhead{}\let\@evenfoot\@oddfoot}
\def\draftdate{\number\month/\number\day/\number\year\ \ \ \hourmin }

\def\citen#1{\if@filesw \immediate\write \@auxout {\string\citation{#1}}\fi%
\@tempcntb\m@ne \let\@h@ld\relax \def\@citea{}%
\@for \@citeb:=#1\do {\@ifundefined {b@\@citeb}%
    {\@h@ld\@citea\@tempcntb\m@ne{\bf ?}%
    \@warning {Citation `\@citeb ' on page \thepage \space undefined}}%
    {\@tempcnta\@tempcntb \advance\@tempcnta\@ne
    \setbox\z@\hbox\bgroup\ifcat0\csname b@\@citeb \endcsname \relax
    \egroup \@tempcntb\number\csname b@\@citeb \endcsname \relax
    \else \egroup \@tempcntb\m@ne \fi \ifnum\@tempcnta=\@tempcntb
    \ifx\@h@ld\relax \edef \@h@ld{\@citea\csname b@\@citeb\endcsname}%
    \else \edef\@h@ld{\hbox{--}\penalty\@highpenalty
    \csname b@\@citeb\endcsname}\fi
    \else \@h@ld\@citea\csname b@\@citeb \endcsname \let\@h@ld\relax \fi}%
\def\@citea{,\penalty\@highpenalty\hskip.13em plus.13em minus.13em}}\@h@ld}
\def\@citex[#1]#2{\@cite{\citen{#2}}{#1}}%
\def\@cite#1#2{\leavevmode\unskip\ifnum\lastpenalty=\z@\penalty\@highpenalty\fi%
  \ [{\multiply\@highpenalty 3 #1%
  \if@tempswa,\penalty\@highpenalty\ #2\fi}]}   %
\makeatother 
\catcode`\@=12

\newcommand\ad[1]  {\mathrm{ad}_{#1}}
\def\auto          {automorphism}
\def\ao            {A_\circ}
\def\av            {A^\vee}
\def\be            {\begin{equation}}
\def\bearl         {\begin{array}{l}}
\def\bearll        {\begin{array}{ll}}
 
\def\C             {\ensuremath{\mathcal C}}

\def\cfts          {conformal field theories}
\def\cir           {\,{\circ}\,}

\def\conp          {\heartsuit}

\def\ee            {\end{equation}}

\def\eear          {\end{array}}
\newcommand\eew[1] {{}^{\wedge}{#1}}
\newcommand\eeW[1] {{}^{\wedge\!}{#1}}
\newcommand\eEW[1] {{}^{\wedge}{\!#1}}
\def\End           {\mathrm{End}}
\def\eps           {\varepsilon}
\def\eq            {\,{=}\,}
\newcommand\eqpic[4]{\begin{eqnarray}
                   \begin{picture}(#2,#3){}\end{picture}\nonumber\\
                   \raisebox{-#3pt}{ \begin{picture}(#2,#3) #4 \end{picture} }
                   \label{#1} \\~\nonumber \end{eqnarray} }
\newcommand\eqpid[4]{\begin{eqnarray}
                   \begin{picture}(#2,#3){}\end{picture}\nonumber\\
                   \raisebox{-#3pt}{ \begin{picture}(#2,#3) #4 \end{picture} }
                   \nonumber \\~\label{#1} \end{eqnarray} }
\newcommand\erf[1] {(\ref{#1})}
\def\Fa            {\FRO\ algebra}
\def\FAs           {$(\Delta,\eps)$-\FRO\ structure}

\def\FBs           {$\kap$-\FRO\ structure}

\def\FCs           {$\irl$-\FRO\ structure}

\def\FRO           {Fro\-be\-ni\-us}
\def\Hom           {\mathrm{Hom}}
\def\id            {\mbox{\sl id}}
\def\ida           {\id_{A}}
\def\idav          {\id_{A^\vee_{\phantom:}}}
\def\idva          {\id_{{}^\vee_{\phantom:}\!\!A}}

\def\ikel          {\Phi_{\kap_\eps,\rm l}}
\def\iker          {\Phi_{\kap_\eps,\rm r}}
\def\ikl           {\Phi_{\kap,\rm l}}
\def\iklm          {\Phi_{\kap,\rm l}^{-1}}
\def\ikpl          {\Phi_{\kap'\!,\rm l}}
\def\ikpr          {\Phi_{\kap'\!,\rm r}}
\def\ikprm         {\Phi_{\kap'\!,\rm r}^{-1}}
\def\ikr           {\Phi_{\kap,\rm r}}
\def\ikrm          {\Phi_{\kap,\rm r}^{-1}}
\def\iN            {\,{\in}\,} 
\newcommand\Includepicffrs[1] {{\begin{picture}(0,0)(0,0)
                   \scalebox{.38}{\includegraphics{imgs/pic_ffrs_#1.eps}}\end{picture}}}
\newcommand\IncludepicfuRsXII[1] {{\begin{picture}(0,0)(0,0)
                   \scalebox{.38}{\includegraphics{imgs/pic_fuRsXII_#1.eps}}\end{picture}}}
\newcommand\IncludepicfuSt[1] {{\begin{picture}(0,0)(0,0)
                   \scalebox{.38}{\includegraphics{imgs/pic_fuSt_#1.eps}}\end{picture}}}
\newcommand\IncludepicjfXXIX[1] {{\begin{picture}(0,0)(0,0)
                   \scalebox{.38}{\includegraphics{imgs/pic_jfXXIX_#1.eps}}\end{picture}}}
\def\irl           {{\Phi_{\!\rhol}}}
\def\irlm          {\Phi_{\!\rhol}^{-1}}
\def\irr           {\Phi_{\!\rhors}}

\def\kap           {\kappa} 
\newcommand\labl[1]{\label{#1}\ee}
\def\lhs           {left hand side}
\newcommand\lo[1]  {\ell_{#1}}
\newcommand\lom[1] {\ell_{#1^{-1}_{}}^{}}
\def\mo            {m_\circ}
\def\naka          {\mho}
\def\Naka          {Nakayama automorphism}
\def\nake          {\mho_\eps}

\def\nondeg        {non-de\-ge\-ne\-ra\-te}
\def\obj           {{\mathcal O}bj}
\def\one           {{\ensuremath{\mathbf 1}}}

\def\oti           {\,{\otimes}\,}
\def\Oti           {{\otimes}}
\def\phil          {\ikel}  
\def\phir          {\iker}  

\long\def\query#1{\hskip 0pt{\vadjust{\everypar={}\small\vtop to 0pt{\hbox{}%
     \vskip -13pt\rlap{\hbox to 49.0pc{\hfil{\vtop{\hsize=8pc\tolerance=6000%
     \hfuzz=.5pc\rightskip=0pt plus 5.5em\noindent#1}}}}\vss}}}}%
\def\rep           {representation}
\def\rhol          {\rho}
\def\rhor          {{\reflectbox{$\rho$}}}
\def\rhors         {{\reflectbox{$\scriptstyle\rho$}}}
\def\rhs           {right hand side}

\newcommand\ro[1]  {r_{\!#1}}
\newcommand\rom[1] {r_{\!#1^{-1}_{}}^{}}
\newcommand\romv[1]{r_{\!#1^{-1}_{}}^{\,\vee}}
\def\sigl          {\sigma_{\rm l}^{}}
\def\sigL          {\sigma_{\rm l}}
\def\sigr          {\sigma_{\rm r}}
\def\sse           {\scriptsize}

\def\Times         {\,{\times}\,}
\def\To            {\,{\to}\,} 
\def\ua            {\ao^\times}
\def\va            {\hspace*{-.8em}{\phantom A}^{\vee\!}\!A}
\newcommand\void[1]{}
\def\Vee           {^{\vee}}
\newcommand\wee[1] {{#1}^{\wedge}}

\documentclass[10pt]{article}
\usepackage{amsthm,latexsym,amssymb,amsfonts,bbm}
\usepackage[mathscr]{eucal}
\usepackage{graphics}
\usepackage{rotating}
\usepackage[pagewise,modulo]{lineno}
\newtheorem{thm}{Theorem}
\newtheorem{lemma}[thm]{Lemma}
\newtheorem{prop}[thm]{Proposition}
\newtheorem{cor}[thm]{Corollary}
\theoremstyle{definition}
\newtheorem{rem}[thm]{Remark}
\newtheorem{Def}[thm]{Definition}
\begin{document}

  {~} \\[-14mm]
\begin{flushright}
{Arabian Journal for Science and Engineering\\ 33-2C (2008) 175\,--\,191} 
\end{flushright}
\vskip 9mm ~\\[1em]~ 
\begin{center}
{\large\bf ON FROBENIUS ALGEBRAS\\[6pt] IN RIGID MONOIDAL CATEGORIES}
\\[14mm]
{\large J\"urgen Fuchs}~~~and~~~{\large Carl Stigner}
\\[6mm]
Teoretisk fysik, \ Karlstads Universitet\\
Universitetsgatan, \ S\,--\,651\,88\, Karlstad
\end{center}
\vskip 29mm

\begin{quote}{\bf Abstract}\\[1mm]
We show that the equivalence between several possible characterizations 
of Frobenius algebras, and of symmetric Frobenius algebras, carries over 
from the category of vector spaces to more general monoidal categories. 
For Frobenius algebras, the appropriate setting is the one of rigid
monoidal categories, and for symmetric Frobenius algebras it is the one of
sovereign monoidal categories. We also discuss 
some properties of Nakayama automorphisms.
\\[1em]
\end{quote}

 \newpage

\section{Introduction}

Frobenius algebras in monoidal categories play a significant role
in diverse contexts. Illustrative examples are the study of
weak Morita equivalences of tensor categories \cite{muge8},
certain correspondences of ribbon categories which give e.g.\ rise to 
the notion of trivializability of a ribbon category \cite{ffrs},
the computation of correlation functions in conformal quantum field theory
\cite{fuRs,fjfrs,scfr2}, the analysis of braided crossed G-categories
\cite{muge13}, the theory of subfactors and of extensions of $C^*$-algebras
\cite{loro,evpi}, invariants of three-dimensional membranes \cite{laud3},
reconstruction theorems for modular tensor categories \cite{pfei5},
and a categorical version of Militaru's D-equation \cite{biStr}.

In the classical case of algebras in the category of vector spaces over a field 
or commutative ring, several equivalent characterizations of Frobenius algebras 
are in use, see e.g.\ \cite[Thm.\,61.3]{CUre}. The most common 
ones are via the existence of an isomorphism between an algebra and its dual
as modules, or via the existence of a \nondeg\ invariant bilinear form.
A more recent description is via the existence of a coalgebra structure with
appropriate compatibility properties \cite{abra,quin}.
Given such a characterization, there will be an analogous notion of
Frobenius algebra in other categories that are monoidal and are equipped
with sufficiently much additional structure.
In view of the applications mentioned above, it is important to know whether
the equivalence between different possible definitions persists in this more
general situation.

Here we establish the equivalence between categorical versions of the three
characterizations of Frobenius algebras just quoted, for the case that the monoidal
category considered is also rigid,\,%
 \footnote{~The requirements that the categories are rigid, respectively 
 sovereign, are not strictly necessary. For details see the Remarks
 \ref{rem1} and \ref{rem2} below.}
i.e.\ has left and right dualities.
Analogous, and more extensive, results have been obtained for Frobenius algebras
in any category of bimodules over a ring (Frobenius extensions) in \cite{KAdi},
and for Frobenius monads, i.e.\ Frobenius algebras in categories of endofunctors 
\cite{lawv3}, in \cite{stre8};
for commutative Frobenius algebras in compact closed categories some of the
results can also be found in \cite{stri3}.
In the vector space case there also exist other characterizations of Frobenius
algebras, such as via ideals and their annihilators (see e.g.\ Theorem 16.40 of 
\cite{LAm}); the study of their categorical versions is beyond the scope of this
note.

Besides being rigid monoidal, no other properties are assumed for the categories
in which these issues are studied. For instance, they need not be linear or 
abelian and need not have direct sums;\,%
  \footnote{~%
  Many, but not all rigid monoidal categories can be embedded as a full
  subcategory in a category of bimodules over a ring, in which case one is
  in the situation studied in \cite{KAdi}.
  Frobenius algebras not covered by this setting are obtained when the category
  does not possess all the generic properties that such full subcategories
  inherit from the category of bimodules.
  An example for a rigid monoidal category which is additive, but not abelian,
  is the category of finitely generated projective modules over a unital
  commutative ring, see e.g.\ \cite[p.\,25]{TUra}. For an example which does
  not admit direct sums, see e.g.\ \cite[p.\,29]{TUra}.
  And a rigid monoidal category that is not even preadditive, with the
  morphism sets not possessing any structure beyond being sets, is the
  two-dimensional cobordism category, whose objects are oriented one-manifolds
  and whose morphisms are cobordisms; a Frobenius algebra in this category is
  given by the circle, with the structural morphisms
  (product, coproduct, unit, counit) just being the elementary cobordisms
  \cite{laud3}. }
but of course, in many applications they do have additional structure, like
being ribbon categories \cite{TUra} or (multi-)fusion categories \cite{etno}.
In any rigid monoidal category there is an abundance of Frobenius algebras:
for any object $X$, the objects $X\oti X^\vee$ and ${}^\vee\!{X}\oti X$ carry
a natural structure of Frobenius algebra, with the structural morphisms
being expressible through the evaluation and coevaluation morphisms.
Examples for Frobenius algebras of this type are star-autonomous monoidal
categories, regarded as objects in the monoidal bicategory $C\!${\sl at} of
small categories, see \cite[Cor.\,3.3]{stre8}.
These Frobenius algebras are in fact all Morita equivalent to the tensor unit.
A generic source for more general Frobenius algebras is provided by monoidal
categories with nontrivial Picard group; Frobenius algebras that correspond to
subgroups of the Picard group can be classified with the help of abelian group 
cohomology \cite{fuRs9}.

We also discuss the equivalence between categorical versions of the 
notion of symmetric \Fa. For formulating these concepts, we require\,$^1$ 
that the categories in question are rigid monoidal and in addition sovereign,
i.e.\ that the left- and right-duality functors coincide.
Afterwards we introduce the notion of Nakayama automorphisms and study some of 
their properties. We show e.g.\ that a \Fa\ is symmetric iff its \Naka s are
inner \auto s. We can then finally expose the relation between any two 
Frobenius structures on an algebra in a rigid monoidal category.

The three notions of \Fa\ are presented in Section 2, and their equivalence is
proven in Section 3. Section 4 is devoted to an analogous discussion of 
symmetric \Fa s. Nakayama automorphisms are studied in Section 5.



\section{Frobenius algebras}

\subsubsection*{Algebras and coalgebras in monoidal categories}

Let $\C\eq(\C,\otimes,\one)$ be a monoidal category.
Without loss of generality we assume \C\ to be strict.
A (unital, associative) {\em algebra\/} (or monoid) $A\eq (A,m,\eta)$ in
\C\ consists of an object $A\iN\obj(\C)$ and morphisms 
$m \iN \Hom(A\oti A,A)$ and $\eta \iN \Hom(\one,A)$ satisfying 
$m\cir(m\oti\ida) \eq m\cir(\ida\oti m)$ and
$m\cir(\eta\oti\ida) \eq \ida \eq m\cir(\ida\oti\eta)$.

We will freely use the graphical notation for morphisms of strict
monoidal categories as described e.g.\ in \cite{joSt6,MAji,KAss} and
\cite{ffrs,fjfrs}. Thus we write
  \eqpic{} {400} {28} {
  \put(20,0)     {\IncludepicfuRsXII{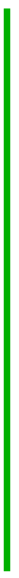}} 
  \put(-23,30)   {$ \id_U^{}\,= $}
  \put(18.0,-8.8){\sse$ U $}
  \put(18.5,65.5){\sse$ U $}
\put(87,0){
  \put(20,0)     {\IncludepicfuRsXII{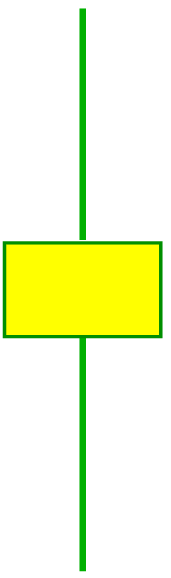}}
  \put(-16,30)   {$ f~= $}
  \put(25.3,-8.8){\sse$ U $}
  \put(26.2,65.5){\sse$ V $}
  \put(26.6,30.2){\sse$ f $}
}
\put(204,0){
  \put(20,0)     {\IncludepicfuRsXII{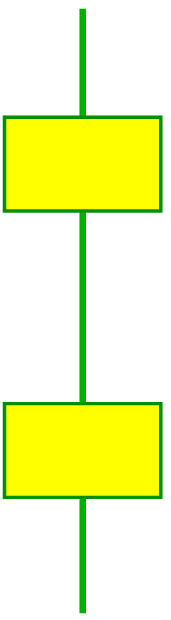}}
  \put(-31,30)   {$ g\cir f~= $}
  \put(25.1,69.5){\sse$ W $}
  \put(25.3,-8.8){\sse$ U $}
  \put(26.6,17.6){\sse$ f $}
  \put(26.6,49.5){\sse$ g $}
  \put(30.5,32.9){\sse$ V $}
}
\put(330,0){
  \put(20,0)     {\IncludepicfuRsXII{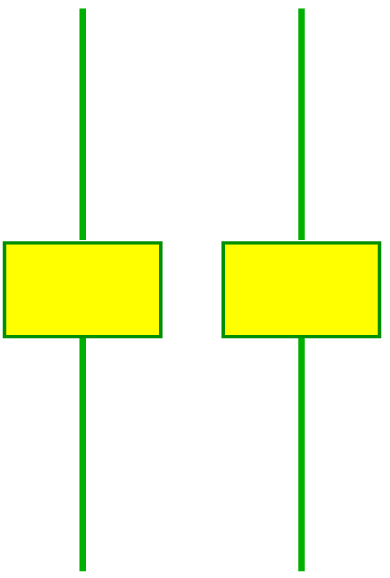}}
  \put(-40,30)   {$ f\oti f'~= $}
  \put(25.3,-8.8){\sse$ U $}
  \put(26.2,65.5){\sse$ V $}
  \put(26.6,29.4){\sse$ f $}
  \put(49.1,-8.8){\sse$ U' $}
  \put(50.0,65.5){\sse$ V' $}
  \put(49.6,29.4){\sse$ f' $}
} }
for identity morphisms, general morphisms $f \iN \Hom(U,V)$, 
and for composition and tensor product of morphisms of \C\
(all pictures are to be read from bottom to top), as well as
  \eqpic{a21} {150} {15} {
\put(10,-2){
  \put(10,0)     {\IncludepicfuRsXII{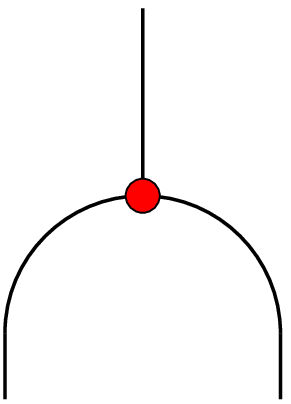}}
  \put(-29,20)   {$ m~= $}
  \put(6.0,-8.8) {\sse$ A $}
  \put(22.5,46.5){\sse$ A $}
  \put(36.0,-8.8){\sse$ A $}
}
\put(140,4){
  \put(10,0)     {\IncludepicfuRsXII{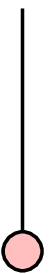}}
  \put(-25,14)   {$ \eta~= $}
  \put(9.5,32.2) {\sse$ A $}
} }
for the product and unit morphisms of an algebra $A$ (note that the
morphism $\id_\one$ is `invisible', owing to strictness of \C).
The defining properties of $A$ then read
  \eqpic{jf29-07} {340} {24} {
\put(0,3){
  \put(25,0)     {\IncludepicjfXXIX{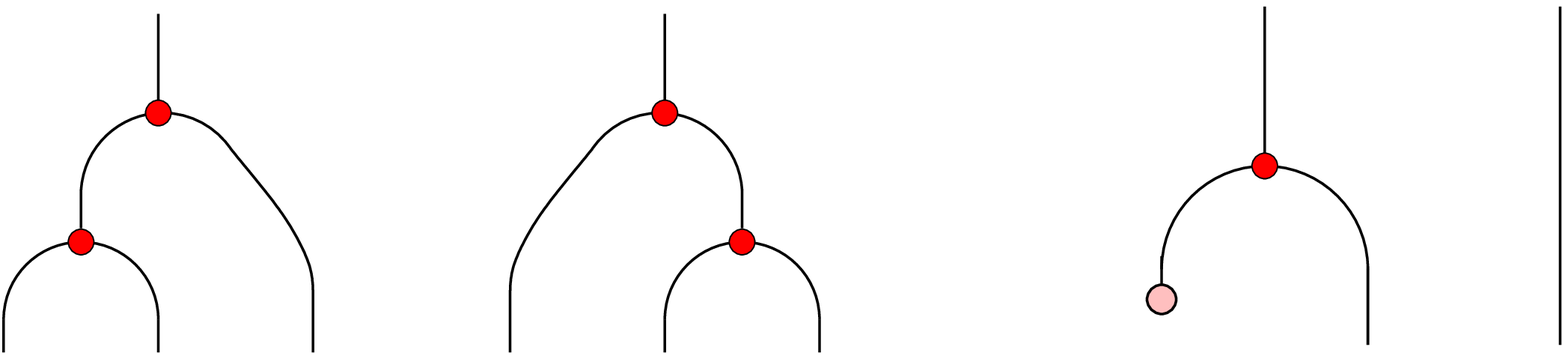}}
  \put(20.8,-9.7){\sse$ A $}
  \put(43.3,-9.7){\sse$ A $}
  \put(45.0,53.8){\sse$ A $}
  \put(66.2,-9.7){\sse$ A $}
  \put(80.2,23)  {$=$}
  \put(95.3,-9.7){\sse$ A $}
  \put(118,-9.7){\sse$ A $}
  \put(119.5,53.8){\sse$ A $}
  \put(140.7,-9.7){\sse$ A $}
  \put(207.2,54.8){\sse$ A $}
  \put(221.2,-8.7){\sse$ A $}
  \put(233.2,23)  {$=$}
  \put(249.6,-8.7){\sse$ A $}
  \put(250.8,54.8){\sse$ A $}
  \put(265.4,23)  {$=$}
  \put(278.9,-8.7){\sse$ A $}
  \put(295.1,54.8){\sse$ A $}
} }

We will also need the dual notion of a (coassociative, counital) coalgebra in 
\C. This is a triple $(C,\Delta,\eps)$ consisting of an object $C$ and morphisms
  \eqpic{a23} {140} {15} {
\put(20,-3){
  \put(10,0)     {\IncludepicfuRsXII{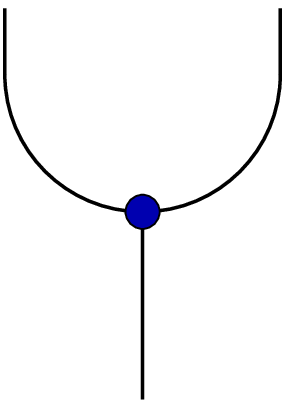}}
  \put(-29,20)   {$ \Delta~= $}
  \put(7.5,46.5) {\sse$ C $}
  \put(21.3,-8.8){\sse$ C $}
  \put(37.6,46.5){\sse$ C $}
}
\put(134,4){
  \put(10,5)     {\IncludepicfuRsXII{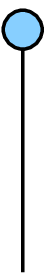}}
  \put(-25,13.2)   {$ \eps~= $}
  \put(8.1,-3.6) {\sse$ C $}
} }
satisfying
  \eqpic{app63,app64} {340} {22} {
\put(20,2){
  \put(10,0)     {\Includepicffrs{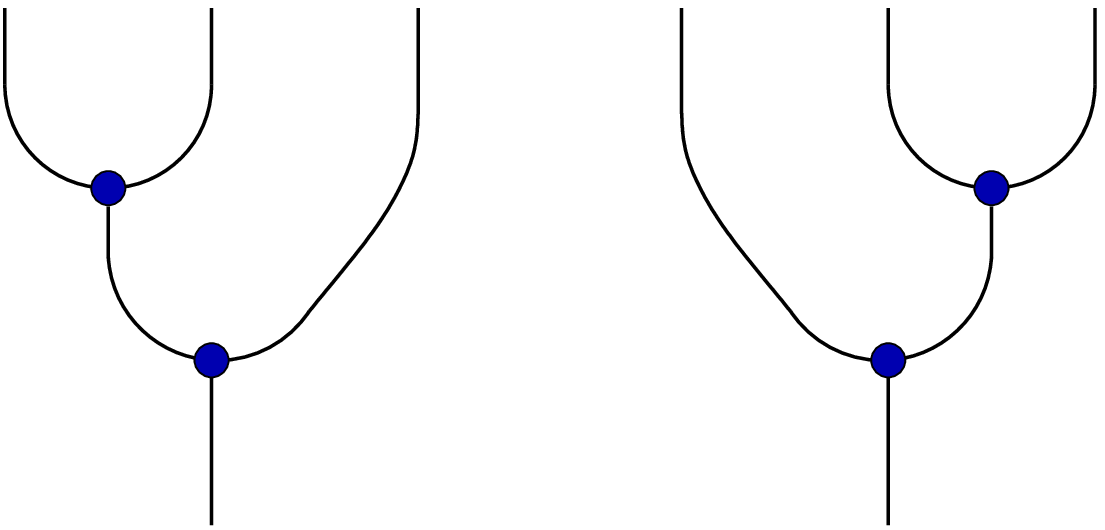}}
\put(7.8,60.4)   {\sse$C$}
\put(28.7,-9.2)  {\sse$C$}
\put(30.3,60.4)  {\sse$C$}
\put(53.0,60.4)  {\sse$C$}
\put(64.5,28)    {\small$=$}
\put(82.0,60.4)  {\sse$C$}
\put(103.1,-9.2) {\sse$C$}
\put(104.6,60.4) {\sse$C$}
\put(127.3,60.4) {\sse$C$}
}
\put(184,0){
  \put(10,0)     {\Includepicffrs{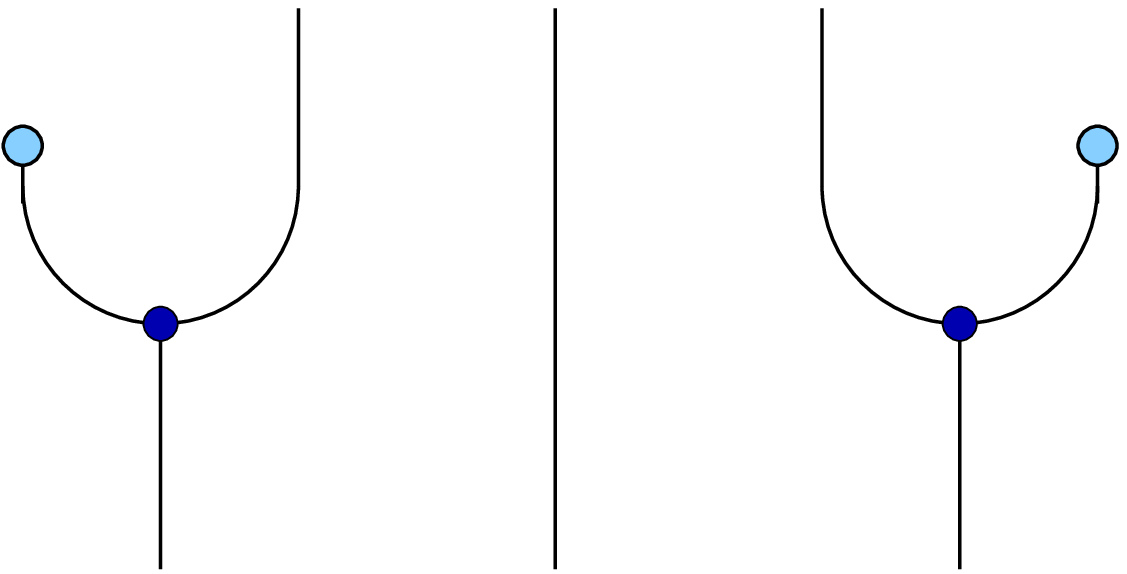}}
\put(23.2,-9.2)  {\sse$C$}
\put(39.8,65.4)  {\sse$C$}
\put(50.5,28)    {\small$=$}
\put(66.5,-9.2)  {\sse$C$}
\put(68.2,65.4)  {\sse$C$}
\put(81.5,28)    {\small$=$}
\put(97.5,65.4)  {\sse$C$}
\put(110.9,-9.2) {\sse$C$}
} }
For more details, see e.g.\ appendix A of \cite{ffrs}.

\subsubsection*{Three definitions of Frobenius algebras}

The classical characterizations of \Fa s mentioned in the introduction
require the {\em existence\/} of some extra structure -- an isomorphism of
$A$-modules, a bilinear form, or a coalgebra structure -- for a given algebra.
In the Definitions \ref{FA}, \ref{FB} and \ref{FC} below we prefer to specify 
instead a choice of the relevant structure explicitly; this will allow 
us to present our arguments in a somewhat more direct manner. A formulation in 
terms of the existence of the extra structures is given in Definition \ref{FRO},
after the equivalence of the three definitions has been established.

There is one notion of Frobenius algebra in a category \C\ that does not 
require any further structure on \C\ beyond what is needed to define algebras,
i.e.\ monoidality:

\begin{Def}\label{FA}
A {\em \FAs\/} on an algebra $(A,m,\eta)$ in a monoidal category
\C\ is a pair $(\Delta,\eps)$ of morphisms such that $(A,\Delta,\eps)$ is a 
coalgebra and the coproduct $\Delta$ is a morphism of $A$-bimodules.
\end{Def}

Here the bimodule structures on $A$ and on $A\oti A$ are the obvious ones
for which the left and right \rep\ morphisms are furnished by the product $m$.
In pictures, the bimodule morphism property of $\Delta$ reads
  \eqpic{picfrob} {280} {31} {
\put(20,-1){
  \put(10,0)     {\Includepicffrs{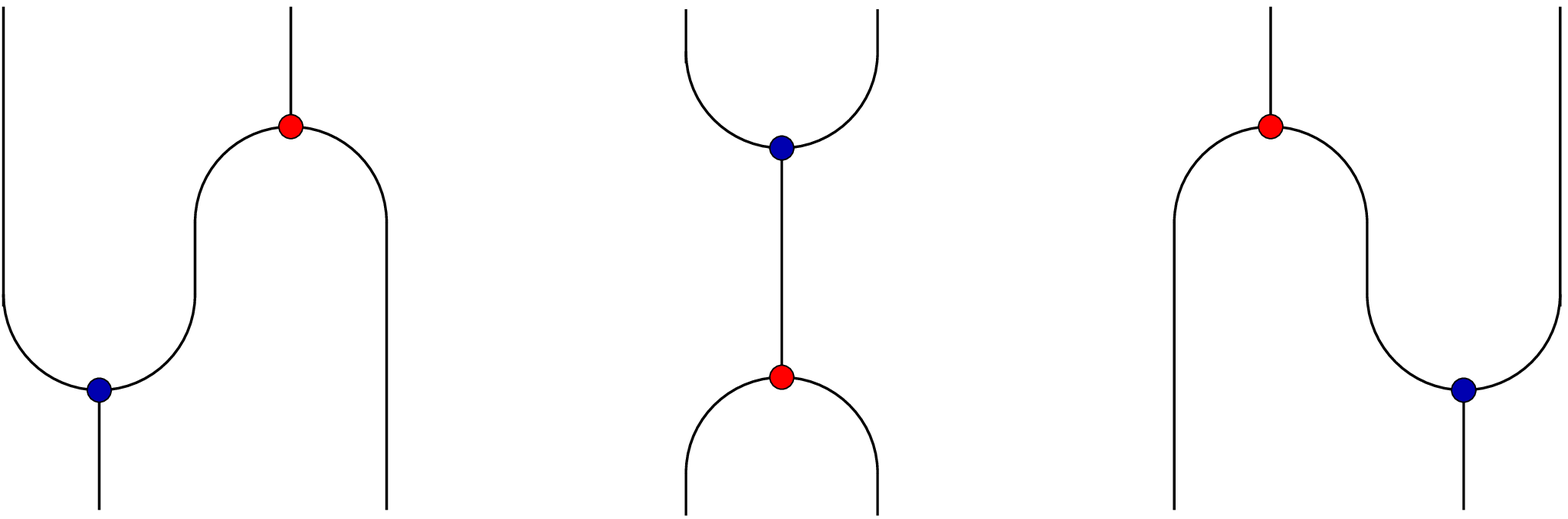}}
\put( 7.3,76.9)  {\sse$A$}
\put(20.5,-9.2)  {\sse$A$}
\put(48.5,76.9)  {\sse$A$}
\put(61.6,-9.2)  {\sse$A$}
\put(83.5,32)    {\small$=$}
\put(104.6,-9.2) {\sse$A$}
\put(105.4,76.9) {\sse$A$}
\put(132.0,-9.2) {\sse$A$}
\put(132.8,76.9) {\sse$A$}
\put(152.5,32)   {\small$=$}
\put(174.5,-9.2) {\sse$A$}
\put(189.0,76.9) {\sse$A$}
\put(215.8,-9.2) {\sse$A$}
\put(230.4,76.9) {\sse$A$}
} }
Note that, unlike in the case of bialgebras (which can be defined in any
braided monoidal category), neither the coproduct $\Delta$
nor the counit $\eps$ is an algebra morphism.

\begin{rem}
The characterization of \Fa s given in Definition \ref{FA} is e.g.\ used in 
\cite{fuRs,muge8,biStr} and implicitly in \cite{kios}.
Besides requiring nothing else than monoidality of \C, it has proved to be 
convenient for performing graphical calculations, and has been excessively used 
for this purpose in e.g.\ \cite{fjfrs,ffrs,fuRs11}. Also, it is this definition
that can readily be generalized to so-called non-compact Frobenius algebras, 
which have recently been discussed in the context of string topology 
\cite{glsux,chMe}.
Alternative descriptions close to the one in Definition \ref{FA} have been 
discussed, and been shown to be equivalent to it, in \cite{laud3}.
\\
The Definition \ref{FA} has been given in \cite{abra} for the category of vector
spaces over a commutative ring, and in \cite{stri3} for compact closed (and
thus in particular symmetric monoidal) categories. More precisely, in 
\cite{abra,stri3} it is further
assumed that $A$ is commutative; as a consequence, the requirement
that the coproduct is a bimodule morphism is equivalent to requiring that it
is a morphism of left (or right) modules. A similar definition, again for vector
spaces, is given in \cite[app.\,A.3]{quin}, where in addition to the bimodule
morphism property of $\Delta$ it is required that the morphism $\eps\cir m$ is 
symmetric; in \cite{quin} the resulting structure is called an {\em ambialgebra\/}.
In the present setting no extra properties are imposed; thus in particular there
is no need that the category \C\ is symmetric, nor even braided.
\end{rem}

The next definition generalizes the familiar one in terms of a bilinear form.
To formulate it, we will have to assume that \C\ is in addition 
{\em rigid\/}, i.e.\ has left- and right-duality endofunctors. We denote the left
and right dual of an object $U$ by $\Vee U$ and $U^\vee$, respectively,\,%
  \footnote{~In our conventions we follow \cite{fuRs}; in most of the literature,
  what we refer to as a left duality is called a right duality, and vice versa.}
and the corresponding evaluation and coevaluation morphisms by
  \eqpic{pic_navf_b} {360} {19} {
\put(0,0){
  \put(10,0)     {\IncludepicfuRsXII{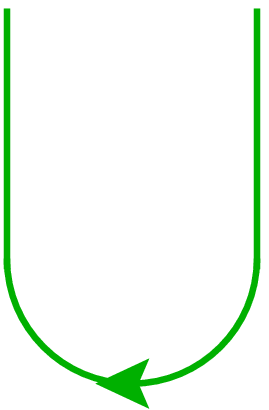}} 
  \put(-32,22)   {$ b_U~= $}
  \put(8.3,47.3) {\sse$ U $}
  \put(33.9,47.3){\sse$ U^\vee $}
}
\put(110,6){
  \put(10,0)     {\IncludepicfuRsXII{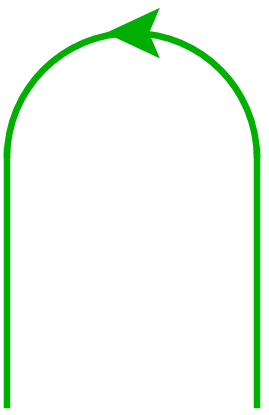}} 
  \put(-32,16)   {$ d_U~= $}
  \put(6.2,-8.2) {\sse$ U^\vee $}
  \put(34.8,-8.2){\sse$ U $}
}
\put(220,0){
  \put(10,0)     {\IncludepicfuRsXII{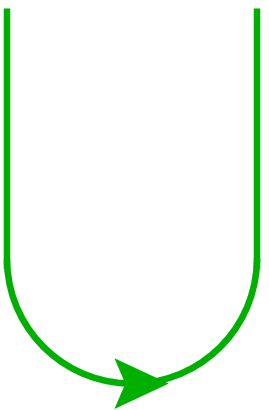}} 
  \put(-32,22)   {$ \tilde b_{U}^{}~= $}
  \put(-2.4,47.3){\sse$ {\phantom U}^\vee_{}\!{U} $}
  \put(34.5,47.3){\sse$ U $}
}
\put(330,6){
  \put(10,0)     {\IncludepicfuRsXII{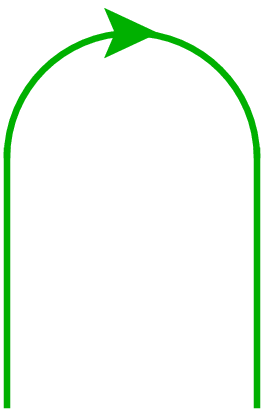}} 
  \put(-32,16)   {$ \tilde d_{U}^{}~= $}
  \put(7.8,-8.2) {\sse$ U $}
  \put(24.9,-8.2){\sse$ {\phantom U}^\vee_{}\!{U} $}
} }

In the sequel we will refer to a morphism in $\Hom(A\oti A,\one)$ as a
{\em pairing on\/} $A$.

\begin{Def}\label{FB}
A {\em \FBs\/} on an algebra $(A,m,\eta)$ in a rigid monoidal category \C\ is 
a pairing $\kap\iN\Hom(A\oti A,\one)$ on $A$ that is {\em invariant\/},
i.e.\ satisfies
  \be
  \kap \circ (m\oti\ida) = \kap \circ (\ida\oti m) \,,
  \ee
and that is {\em \nondeg\/} in the sense that
  \be
  (\idva\oti\kap)\circ(\tilde b_A\oti\ida)~\iN\Hom(A,\va)
  \ee
is an isomorphism.
\end{Def}

In pictures, denoting the pairing $\kap$ by
  \eqpic{pic_csp_01} {1} {1} {
\put(0,0){
  \put(-28,8)    {$\kap \,=: $}
  \put(10,0)     {\IncludepicfuSt{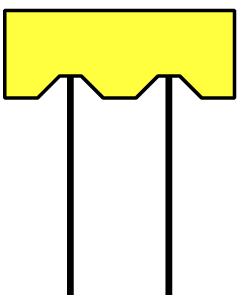}}
  \put(14,-10)   {\sse$A$}
  \put(25,-10)   {\sse$A$}
} }
the invariance property reads
  \eqpic{pic_csp_02} {140} {18} {
\put(0,0){
  \put(10,0)     {\IncludepicfuSt{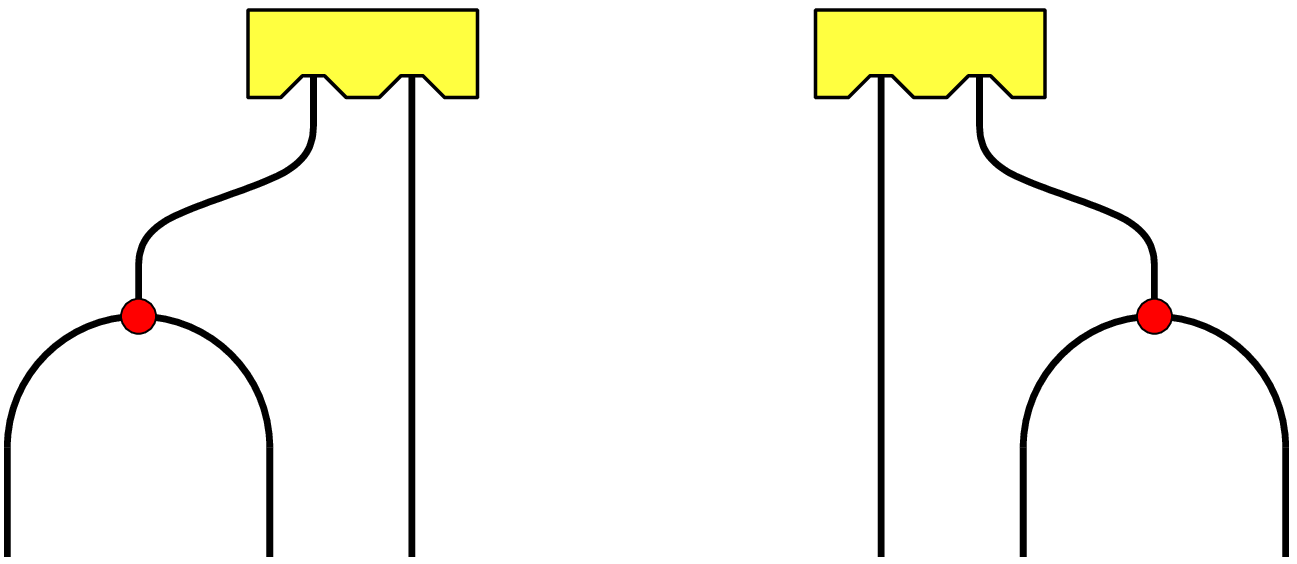}}
  \put(77,25)    {$=$}
  \put(6.4,-10)  {\sse$A$}
  \put(35.2,-10) {\sse$A$}
  \put(50.9,-10) {\sse$A$}
  \put(102.4,-10){\sse$A$}
  \put(118.2,-10){\sse$A$}
  \put(147.0,-10){\sse$A$}
} }
while the isomorphism featuring in the non-degeneracy property is depicted as 
  \eqpic{pic_csp_03} {80} {26} {
\put(0,-3){
  \put(10,0)     {\IncludepicfuSt{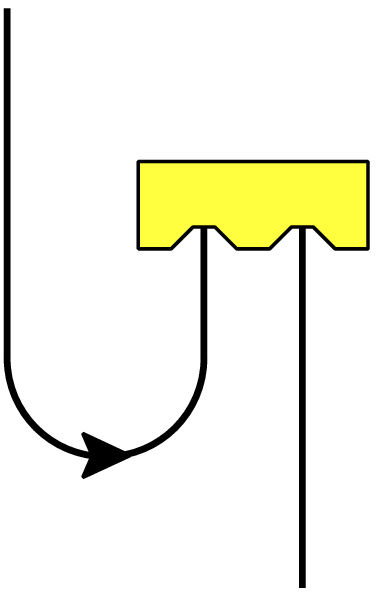}}
  \put(66,25)    {$=:\, \ikl$}
  \put(7,69.2)   {\sse$\va$}
  \put(39,-9.5)  {\sse$A$}
} }

We note that instead of $\ikl$ one may as well use the morphism
  \eqpic{pic_csp_04} {300} {25} {
\put(0,-4){
  \put(0,25)     {$ \ikr \,:=\, (\kap\oti\idav)\circ(\ida\oti b_A) \,= $}
  \put(180,0)    {\IncludepicfuSt{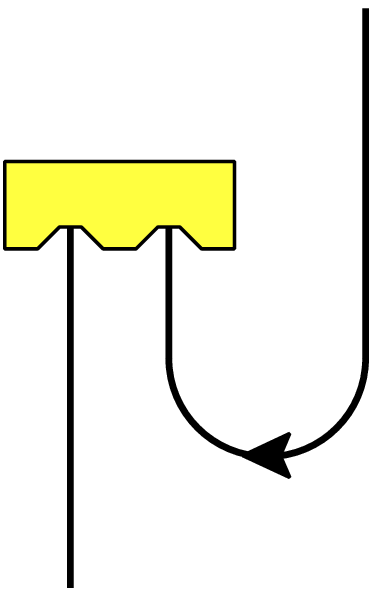}}
  \put(183,-10)  {\sse$A$}
  \put(216,70)   {\sse$\av$}
  \put(230,25)   {$ \iN\Hom(A,\av) $}
} }
Indeed, we have
  \be
  d_A \circ (\ikr\oti\ida) = \kap = \tilde d_A \circ (\ida\oti\ikl) 
  \labl{d ikr = kap}
and a similar identity involving the coevaluation morphisms
$b_A$ and $\tilde b_A$; as a consequence the invertibility of 
$\ikr$ is equivalent to the invertibility of $\ikl$.

Next we note that for $U,V\iN\obj(\C)$ we can associate to any
$f\iN\Hom(U,{}^{\vee\!}V)$ a morphism
  \be
  \wee f := (\tilde d_V\oti \id_{U^\vee}) \circ
  (\id_V \oti f \oti \id_{U^\vee}) \circ (\id_V \oti b_U)
  \,\iN \Hom(V,U^\vee) \,.
  \labl{wee}
The morphism $\wee f$ is an isomorphism iff $f$ is, in which case its inverse 
is given by $(\wee f)^{-1}\eq (d_U\oti\id_V) \cir(\id_{U^\vee}\oti f^{-1}\oti
  $\linebreak[0]$
\id_V) \cir (\id_{U^\vee} \oti\tilde b_V)$ 
(this shows e.g.\ that an object $U$ of \C\ is isomorphic to $^{\vee\!}U$ iff it
  \pagebreak[0]
is isomorphic to $U^\vee$). Using the notation \erf{wee}, the equality 
  \be
  \ikr = \wee{(\ikl)}
  \labl{ikr=iklwee}
is equivalent to \erf{d ikr = kap}.
For convenience we also note some other immediate consequences of 
\erf{d ikr = kap} and its analogue for $b_A$ and $\tilde b_A$: we have
  $ d_A \cir (\idav\oti\ikl^{-1}) \eq \tilde d_A\cir (\ikr^{-1}\oti\idva)
  \,\iN \Hom(\av\Oti\,\va,\one) $, 
as well as
  \be
  (\ikl^{}\oti\ikrm) \circ b_A = \tilde b_A \qquad{\rm and}\qquad
  d_A \circ (\ikr^{}\oti\iklm) = \tilde d_A \,.
  \labl{ikl x ikrm b = tb} 
Let us display the latter identities also pictorially:
  \eqpic{pic_csp_05} {330} {19} {
\put(0,-11){
  \put(0,0)      {\IncludepicfuSt{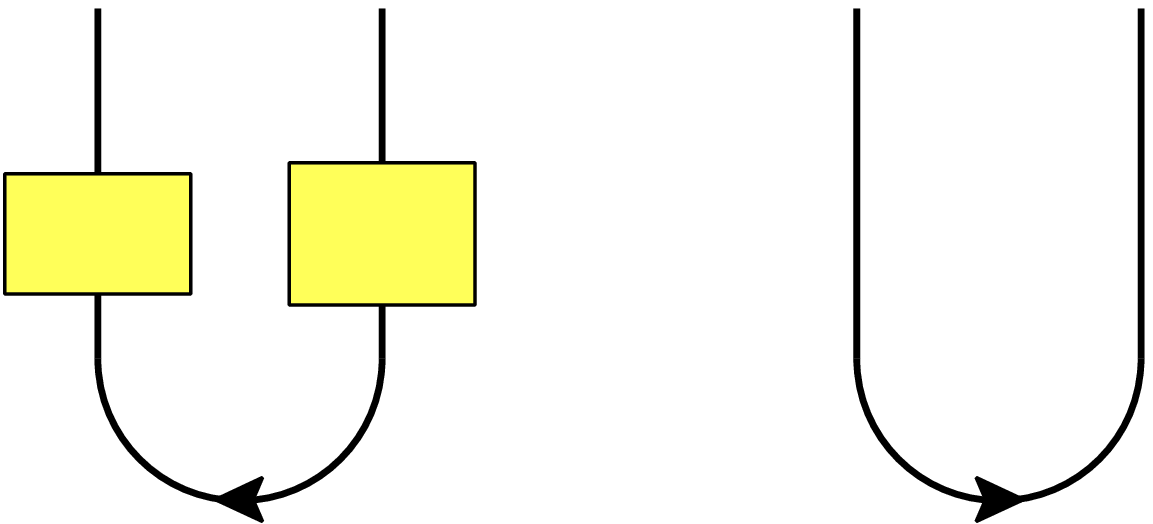}}
  \put(0.5,29)   {$\ikl$}
  \put(31.7,29.0)  {$\ikrm$}
  \put(67,28)    {$ = $}
  \put(7.2,61){\sse$\va$}
  \put(38.2,61){\sse$A$}
  \put(90.5,61){\sse$\va$}
  \put(121.5,61){\sse$A$}
  \put(160,28)   {and}
  \put(215,10)   {\IncludepicfuSt{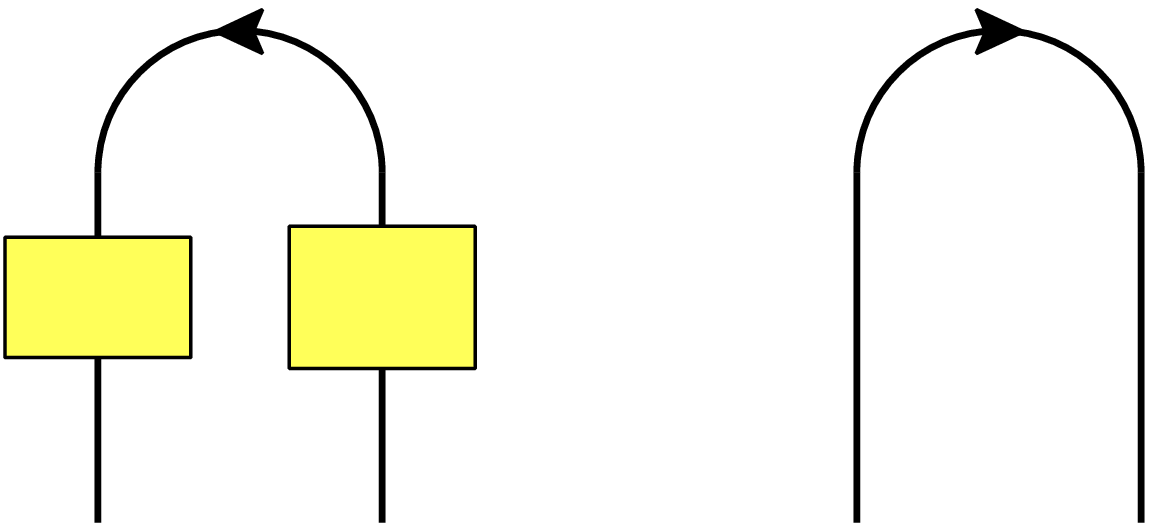}}
  \put(221.6,0)  {\sse$A$}
  \put(251.2,0)  {\sse$\va$}
  \put(304.8,0)  {\sse$A$}
  \put(334.4,0)  {\sse$\va$}
  \put(215.5,32.5) {$\ikr$}
  \put(246.5,32.5) {$\iklm$}
} }

\medskip

Next we note that the left dual $\va$ of $A$ is naturally a left module over 
$A$, while the right dual $\av$ is naturally a right $A$-module.\,%
  \footnote{~%
  In contrast, there is no natural right $A$-module structure on $\va$ unless
  the double dual ${}^{\vee\vee\!}\!A$ is equal to $A$, and similarly for $\av$.}
The corresponding \rep\ morphisms $\rhol\iN\Hom(A\oti\va,\va)$ and
$\rhor\iN\Hom(\av\oti A,\av)$ are given by
  \eqpic{rhol,rhor} {290} {38} {
\put(0,-2){
  \put(0,43)    {$ \rhol \,= $}
  \put(38,0)    {\IncludepicfuSt{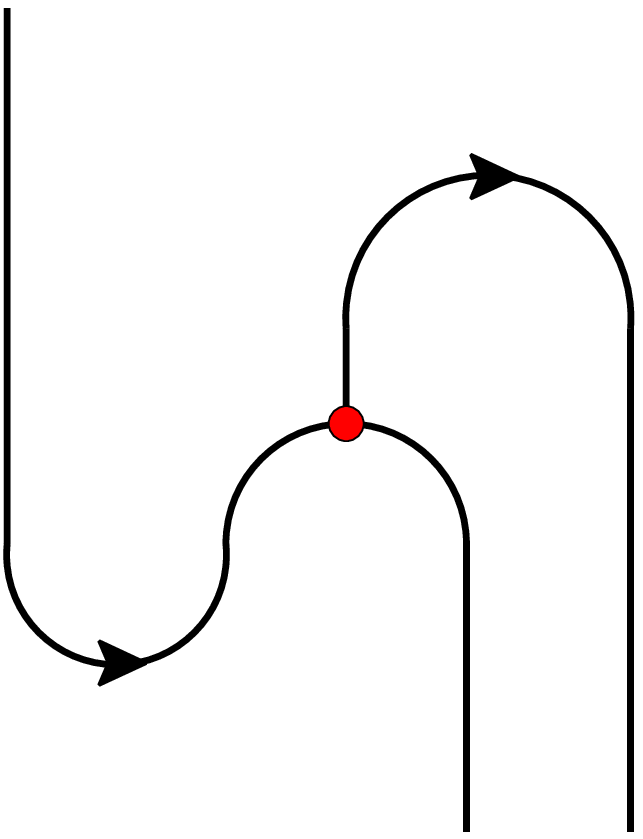}}
  \put(180,43)  {$ \rhor \,= $}
  \put(218,0)   {\IncludepicfuSt{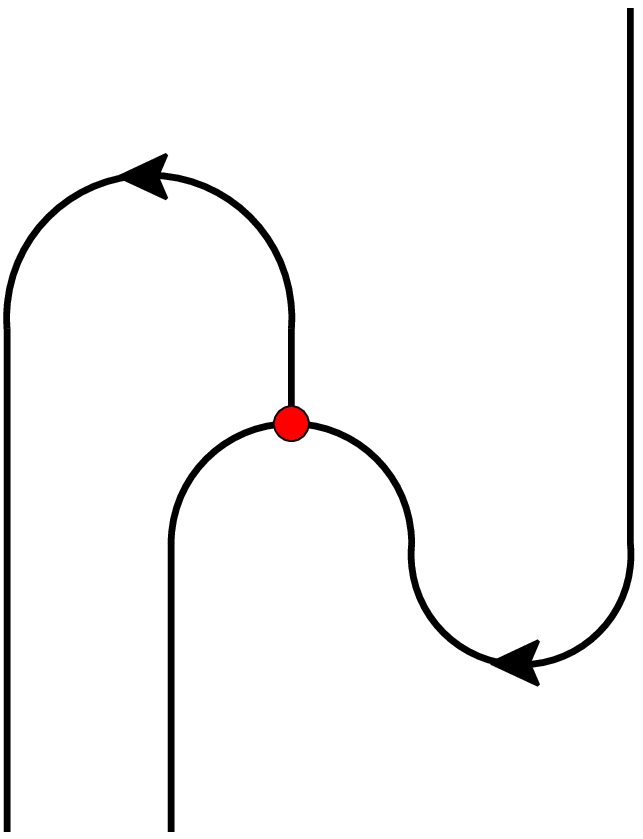}}
  \put(35.3,96) {\sse$\va$}
  \put(85.4,-10) {\sse$A$}
  \put(102.8,-10){\sse$\va$}
  \put(214.5,-10){\sse$\av$}
  \put(232.6,-10){\sse$A$}
  \put(282.8,96) {\sse$\av$}
} }
It is therefore natural to wonder whether the isomorphisms \erf{pic_csp_03} and
\erf{pic_csp_04} are compatible with the $A$-module structures of $A$, $\va$ and
$\av$. This leads to the

\begin{Def}\label{FC}
A {\em \FCs\/} on an algebra $A\eq(A,m,\eta)$ in a rigid monoidal category \C\ 
is a left-module isomorphism $\irl\iN\Hom(A,\va)$ between the left $A$-modules
$A\eq(A,m)$ and $\va\eq(\va,\rhol)$.
\end{Def}

We hasten to remark that $\va$ is of course not preferred over $\av$. Indeed we 
have

\begin{lemma}\label{lem1}
An algebra $A\eq(A,m,\eta)$ in a rigid monoidal category \,\C\ is isomorphic to
$(\va,\rhol)$ as a left $A$-module iff
$(A,m)$ is isomorphic to $\av\eq(\av,\rhor)$ as a right $A$-module.
\end{lemma}

\begin{proof}
That $\irl\iN\Hom(A,\va)$ is a morphism of left $A$-modules means that
$\irl\cir m \eq \rho\cir(\ida\oti\irl)$, which in turn is equivalent to the
equality
  \eqpic{pic_csp_11} {140} {28} {
 \put(0,3){  
  \put(0,0)     {\IncludepicfuSt{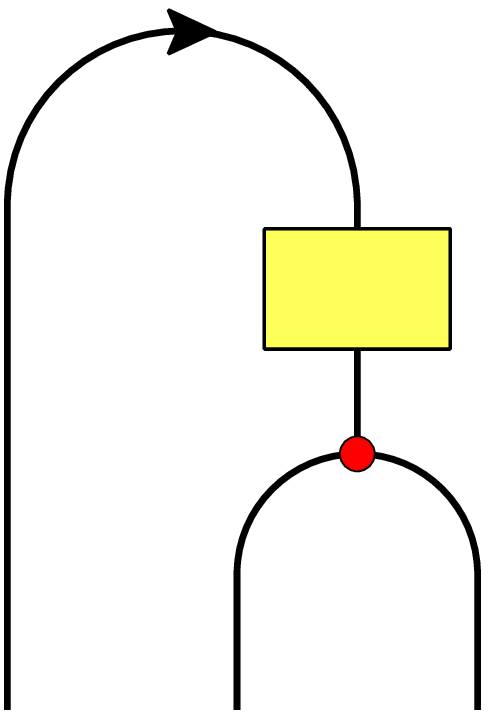}}
  \put(72,36)   {$ = $}
  \put(100,0)   {\IncludepicfuSt{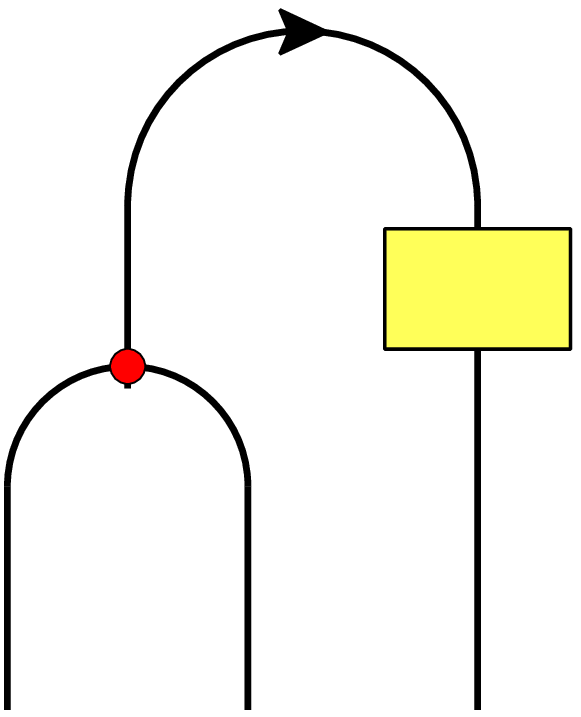}}
  \put(32,44)   {$\irl$}
  \put(144.9,44){$\irl$}
  \put(-3,-10)  {\sse$A$}
  \put(22.2,-10){\sse$A$}
  \put(48.4,-10){\sse$A$}
  \put(96.7,-10){\sse$A$}
  \put(123.2,-10){\sse$A$}
  \put(148.4,-10){\sse$A$}
} }
between morphisms in $\Hom(A\oti A\oti A,\one)$. Now given $\irl$, consider
$\irr\,{:=}\,\wee{\irl}\iN \Hom(A,\av)$
defined according to the prescription \erf{wee}. $\irr$ is invertible if and 
only if $\irl$ is. And the equality
  \eqpic{pic_csp_12} {140} {30} {
    \put(0,2){
  \put(0,0)     {\IncludepicfuSt{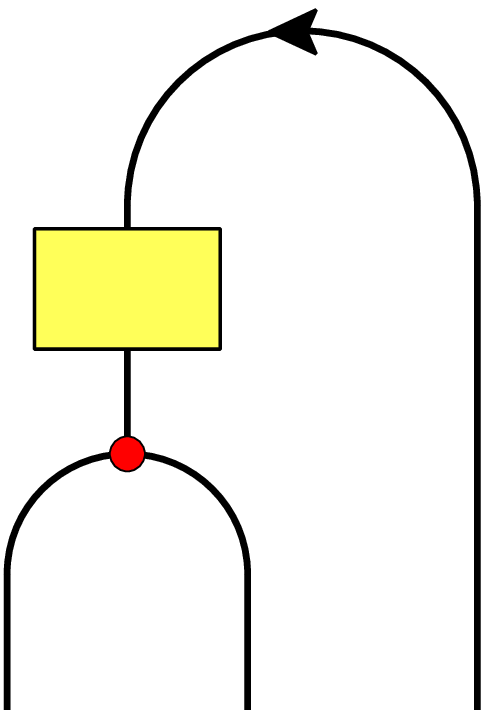}}
  \put(72,36)   {$ = $}
  \put(100,0)   {\IncludepicfuSt{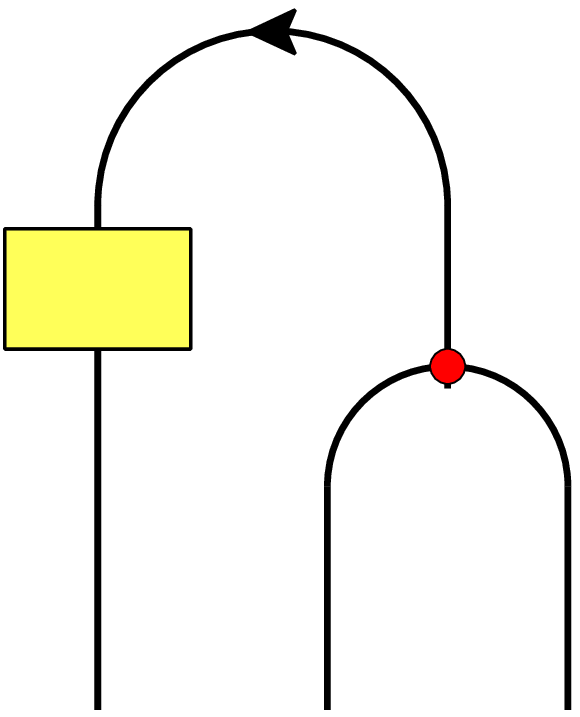}}
  \put(7.2,44)   {$\irr$}
  \put(104.2,44) {$\irr$}
  \put(-3,-10)   {\sse$A$}
  \put(23.2,-10) {\sse$A$}
  \put(48.4,-10) {\sse$A$}
  \put(106.7,-10){\sse$A$}
  \put(132.1,-10){\sse$A$}
  \put(158.4,-10){\sse$A$}
} }
is equivalent to $\irr$ being a morphism of right $A$-modules.
\\[3pt]
Now notice that expressing $\irr$ in terms of $\irl$ and invoking the
defining properties of the evaluation and coevaluation morphisms, one sees that
the \lhs\ of \erf{pic_csp_12} is equal to the \rhs\ of \erf{pic_csp_11}, while
the \rhs\ of \erf{pic_csp_12} is equal to the \lhs\ of \erf{pic_csp_11}.
Thus validity of \erf{pic_csp_12} is equivalent to validity of \erf{pic_csp_11}.
\end{proof}

\begin{rem}\label{rem1}
In the considerations above, the existence of left and right dual objects 
and of the corresponding (co)evaluation morphisms is only needed for the
particular object $A$ under study. Still we prefer to assume the stronger
condition that the category \C\ is rigid, since the presence of such structures 
for $A$ cannot be expected unless
they form a part of suitable left- and right-duality endofunctors of \C.
\end{rem}

\begin{rem}
The fact that $\irlm\iN\Hom(\va,A)$ is a
morphism of left $A$-modules is equivalent to
  \eqpic{pic_csp_19} {140} {32} {
\put(0,-5){
  \put(0,0)     {\IncludepicfuSt{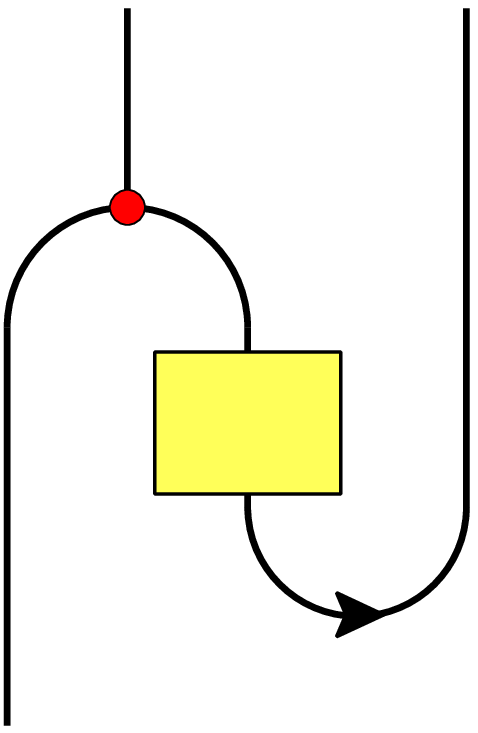}}
  \put(74,36)   {$ = $}
  \put(100,0)   {\IncludepicfuSt{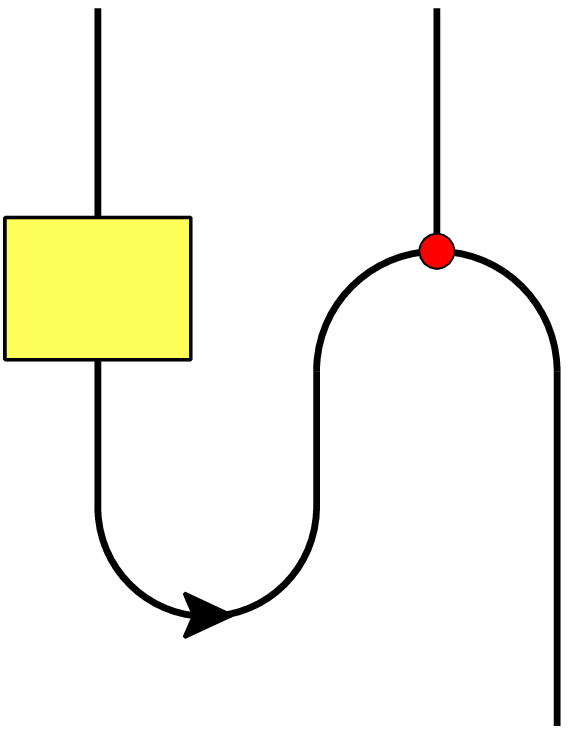}}
  \put(-3,-10)     {\sse$A$}
  \put(9.6,85)     {\sse$A$}
  \put(47,85)      {\sse$A$}
  \put(157.2,-10)  {\sse$A$}
  \put(106.4,85)   {\sse$A$}
  \put(143.8,85)   {\sse$A$}
  \put(17.2,31.2)  {$\irlm$}
  \put(100.6,45.9) {$\irlm$}
} }
As a consequence, the morphism
  \be
  e := (\iklm\oti\ida) \circ \tilde b_A ~\iN\Hom(\one,A\oti A)
  \ee
is invariant in the sense that $(m\oti\ida) \cir (\ida\oti e) \eq
(\ida\oti m) \cir (e\oti\ida)$. If $A$ is such that $e$ in addition satisfies
$m\cir e\eq\eta$, then $e$ is an idempotent with respect to the convolution product 
$\conp$ on $\Hom(\one,A\oti A)$ that is given by
  \eqpic{pic_csp_24} {130} {34} {
\put(0,-8){
  \put(0,40)      {$ \conp(g,h) ~:= $}
  \put(72,0)      {\IncludepicfuSt{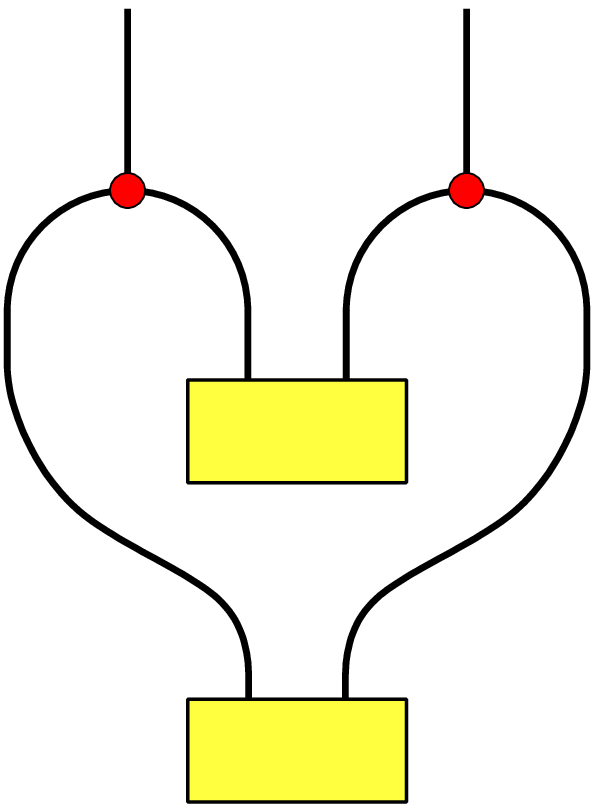}}
  \put(81.5,91)   {\sse$A$}
  \put(118.8,91)  {\sse$A$}
  \put(101,39.7)  {$g$}
  \put(101,2.4)   {$h$}
} }
for $f,g\iN\Hom(\one,A\oti A)$. In fact, $e$ is then a separability 
idempotent for the algebra $A$.
\end{rem}


\section{Equivalence of the three notions of \Fa}

In this section we establish that in rigid monoidal categories the three 
Definitions \ref{FA}, \ref{FB} and \ref{FC} actually
describe one and the same concept.

\begin{prop}
In a rigid monoidal category \,\C\ the notions of \FAs\ and of \FBs\ on an
algebra $(A,m,\eta)$ are equivalent.
\\[3pt]
More concretely:
\\[4pt]
{\rm(i)}\,~~If $(A,m,\eta,\Delta,\eps)$ is an algebra with \FAs, then 
$(A,m,\eta,\kap_\eps)$ with
  \be
  \kap_\eps := \eps \circ m
  \labl{kapepsdef}
is an algebra with \FBs. 
\\[4pt]
{\rm(ii)}\,~If $(A,m,\eta,\kap)$ is an algebra with \FBs, then 
$(A,m,\eta,\Delta_\kap,\eps_\kap)$ with
  \be
  \Delta_\kap := (\ida\oti m) \circ (\ida\oti\ikrm\oti\ida) \circ (b_A\oti\ida)
  \qquad{\rm and}\qquad
  \eps_\kap := \kap \circ (\ida\oti\eta)
  \labl{deltakapepskapdef}
is an algebra with \FAs. 
\end{prop}

\begin{proof}
(i)~~Let $(A,m,\eta,\Delta,\eps)$ be an algebra with \FAs\ in the category \C\ 
and define the 
pairing $\kap_\eps\iN\Hom(A\oti A,\one)$ by \erf{kapepsdef}. Then the morphisms
  \be  
  \ikel = (\idva\oti\kap_\eps)\cir(\tilde b_A\oti\ida)
  \qquad{\rm and}\qquad
  \iker = (\kap_\eps\oti\idav)\cir(\ida\oti b_A) 
  \labl{defphilphir}
are isomorphisms, with inverses given by the morphisms
$(\ida\oti\tilde d_A)\cir((\Delta{\circ}\eta)\oti\idva) \iN\Hom(\va,A)$ and
$(d_A\oti\ida)\cir
  $\linebreak[0]$
(\idav\oti(\Delta{\circ}\eta)) \iN\Hom(\av,A)$, respectively.
Thus $\kap_\eps$ is \nondeg.
That $\kap_\eps$ is invariant is an immediate consequence of the
associativity of the product $m$.
\\
Thus $(A,m,\eta,\kap_\eps)$ is an algebra with \FBs.
\\[3pt]
(ii)\,~Let $(A,m,\eta,\kap)$ be an algebra with \FBs\ in \C\ and define
$\Delta_\kap$ and $\eps_\kap$ by \erf{deltakapepskapdef}.
\\[3pt]
(ii\,a)\,~To see that $\Delta_\kap$ is a coassociative coproduct, first notice that
with the help of invariance of $\kap$, the product $m$ can be rewritten as
  \eqpic{pic_csp_06} {340} {36} {
\put(0,-3){
  \put(0,39)     {$ m \,=\, \ikrm \circ \ikr \circ m \,= $}
  \put(132,0)    {\IncludepicfuSt{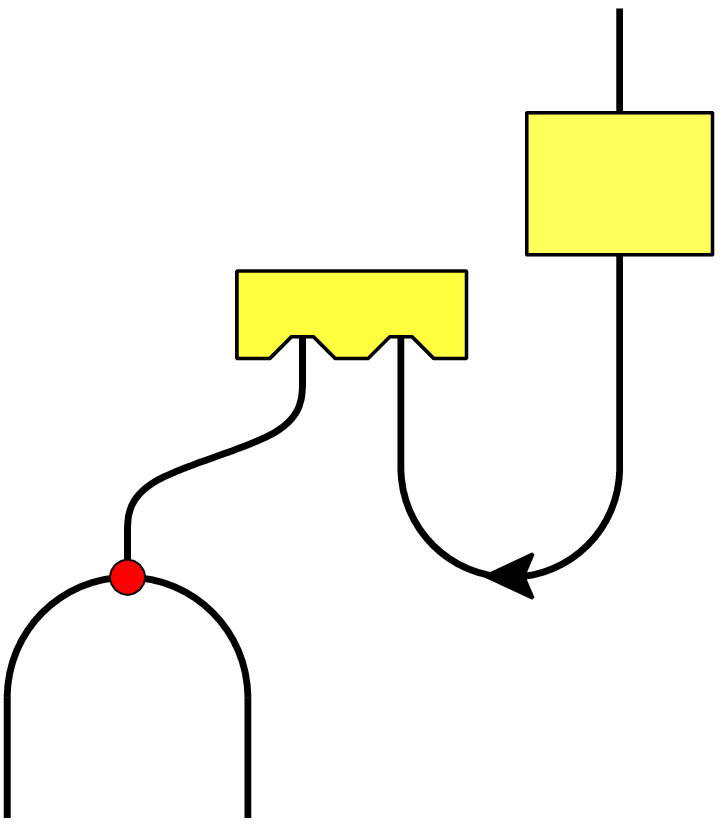}}
  \put(225,39)   {$ = $}
  \put(255,0)    {\IncludepicfuSt{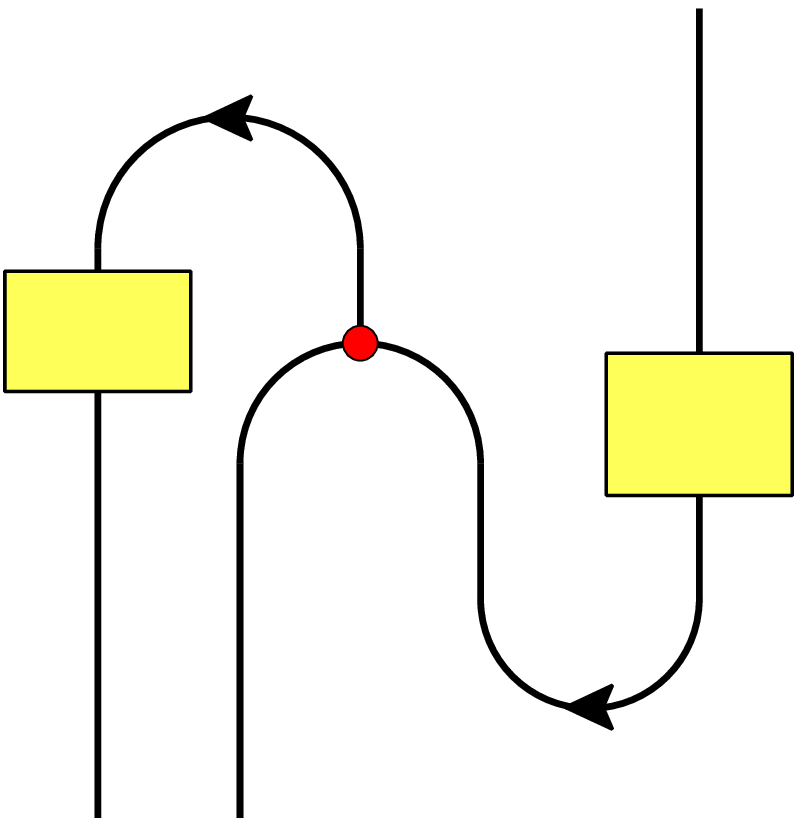}}
  \put(128.3,-10){\sse$A$}
  \put(154.6,-10){\sse$A$}
  \put(196.6,94) {\sse$A$}
  \put(190.0,67.2){$\ikrm$}
  \put(261.2,-10){\sse$A$}
  \put(276.7,-10){\sse$A$}
  \put(328.5,94) {\sse$A$}
  \put(255.4,51) {$\ikr$}
  \put(321.6,41.0){$\ikrm$}
} }
Together with \erf{pic_csp_05} this implies that two alternative descriptions 
of $\Delta_\kap$ are
  \eqpic{pic_csp_07} {340} {38} {
\put(0,-5){
  \put(0,41)      {$ \Delta_\kap \,\equiv $}
  \put(52,0)      {\IncludepicfuSt{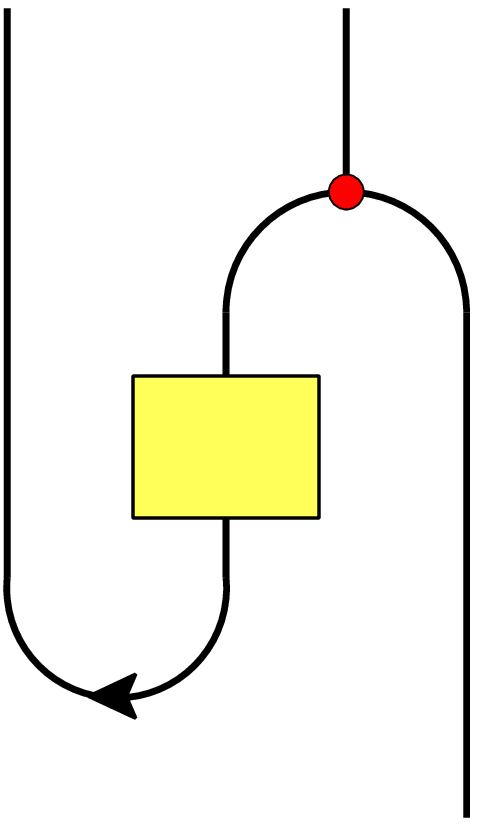}}
  \put(125,41)    {$ = $}
  \put(155,0)     {\IncludepicfuSt{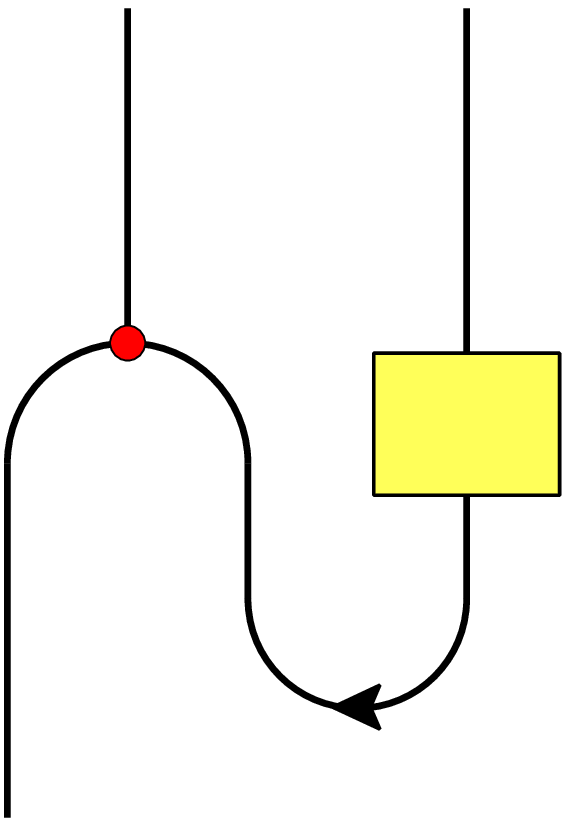}}
  \put(238,41)    {$ = $}
  \put(268,0)     {\IncludepicfuSt{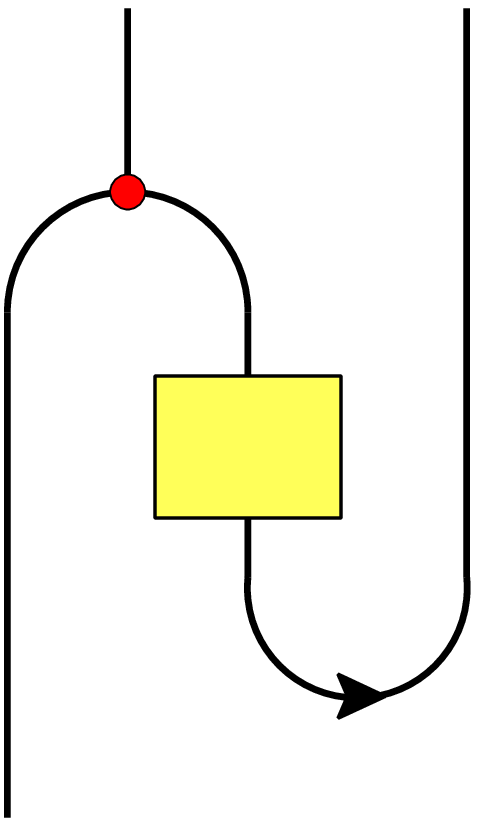}}
  \put(49.2,94.4) {\sse$A$}
  \put(86.3,94.4) {\sse$A$}
  \put(99.0,-8.1) {\sse$A$}
  \put(165.5,94.4){\sse$A$}
  \put(202.6,94.4){\sse$A$}
  \put(151.3,-8.1){\sse$A$}
  \put(278.4,94.4){\sse$A$}
  \put(315.9,94.4){\sse$A$}
  \put(264.6,-8.1){\sse$A$}
  \put(66.9,38.5) {$\ikrm$}
  \put(196.0,41)  {$\ikrm$}
  \put(285,38.7)  {$\iklm$}
} }
Using the first two descriptions we can write
  \eqpic{pic_csp_08} {400} {58} {
\put(0,12){
  \put(-40,51)     {$ (\Delta_\kap\oti\ida) \circ \Delta_\kap \,= ·$}
  \put(68,0)       {\IncludepicfuSt{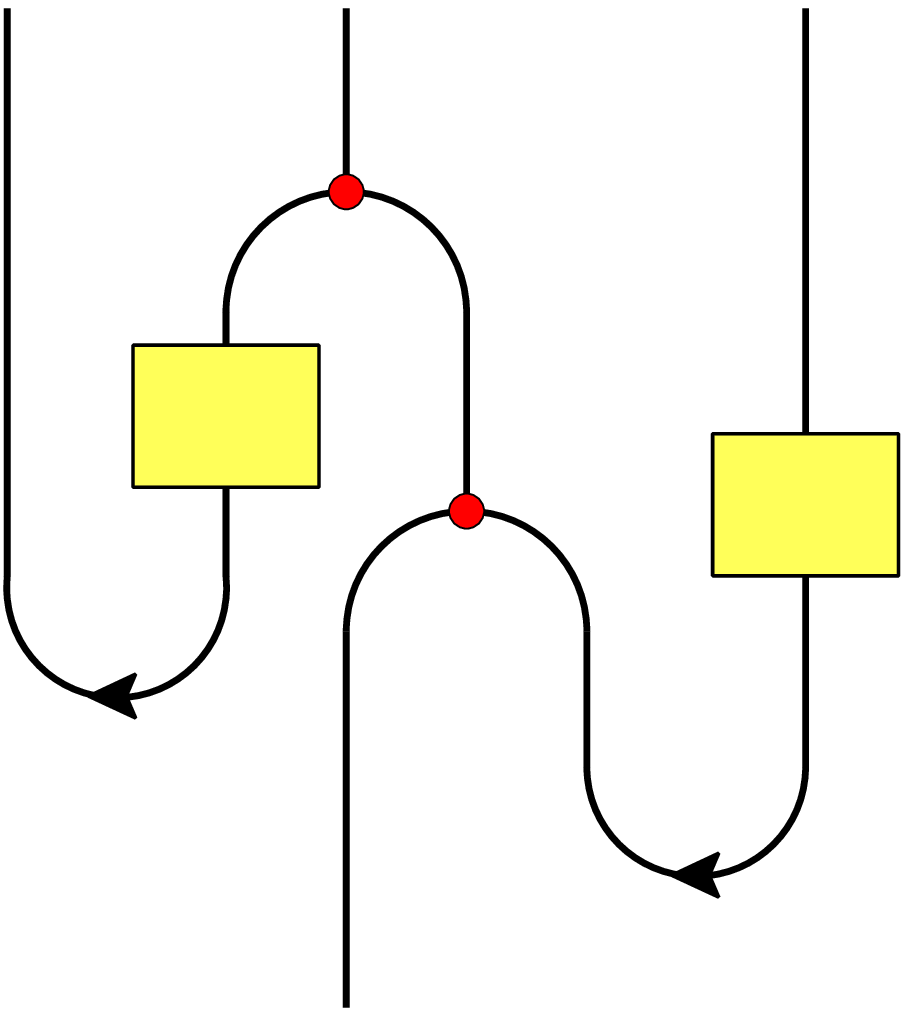}}
  \put(65.5,115.5) {\sse$A$}
  \put(102.5,115.5){\sse$A$}
  \put(153,115.5)  {\sse$A$}
  \put(101.5,-10)  {\sse$A$}
  \put(82.9,62.8)  {$\ikrm$}
  \put(146.6,53.1) {$\ikrm$}
  \put(190,51)     {and $\quad~ (\ida\oti\Delta_\kap) \circ \Delta_\kap \,= $}
  \put(335,0)      {\IncludepicfuSt{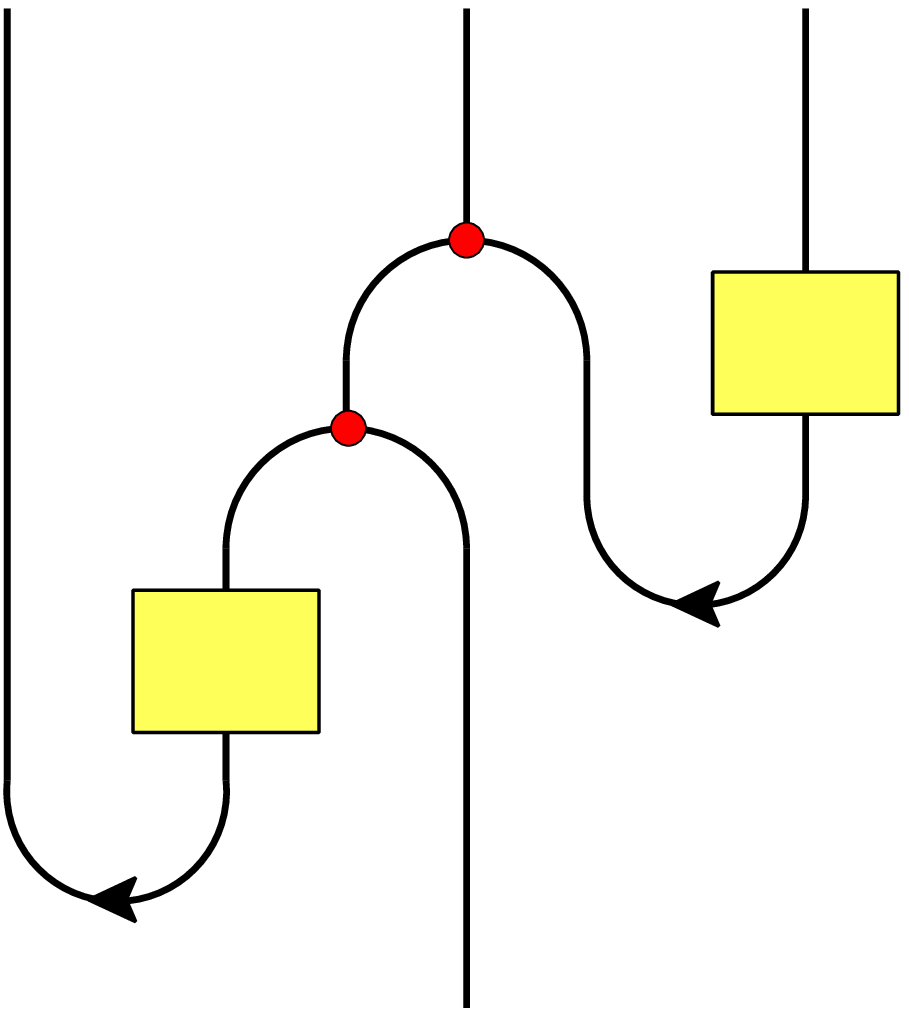}}
  \put(332.4,115.5){\sse$A$}
  \put(383,115.5)  {\sse$A$}
  \put(420.1,115.5){\sse$A$}
  \put(381.8,-10)  {\sse$A$}
  \put(350.0,35.7) {$\ikrm$}
  \put(413.6,70.6) {$\ikrm$}
} }
~\\[-24pt]
Thus coassociativity of $\Delta_\kap$ follows immediately from
associativity of $m$.
\\{}\\[-10pt]
(ii\,b)\,~Next we verify that $\eps_\kap$ is a counit for $\Delta_\kap$. 
First note that, by invariance of $\kap$, an alternative description of 
$\eps_\kap$ is 
  $ \eps_\kap \eq \kap \cir (\eta\oti\ida) $.
Thus together with \erf{d ikr = kap} one obtains
  \be
  \eps_\kap = \tilde
  d_A\circ(\ida\oti(\ikl\circ\eta)) = d_A\circ((\ikr\oti\eta)\oti\ida) \,.
  \labl{epskap = td e x ikl}
Combining the first of these expressions for $\eps_\kap$ with the third 
description of $\Delta_\kap$ in \erf{pic_csp_07} one immediately arrives at 
$(\ida\oti\eps_\kap)\cir\Delta_\kap\eq\ida$, while the
second expression together with the first in \erf{pic_csp_07} yields
$(\eps_\kap\oti\ida)\cir\Delta_\kap\eq\ida.$
\\[3pt]
(ii\,c)\,~To derive the equalities \erf{picfrob} is now easy. Using the first
of the descriptions \erf{pic_csp_07} for $\Delta_\kap$ and associativity,
one sees that
  \be
  \Delta_\kap \circ m = (\ida\oti m)\circ(\Delta_\kap\oti\ida) \,,
  \ee
while 
  \be
  \Delta_\kap \circ m = (m\oti\ida)\cir(\ida\oti\Delta_\kap) 
  \ee
follows by using instead the third of those descriptions and again associativity.
\end{proof}

With the informations gathered so far at hand,
it is straightforward to verify also 

\begin{prop}
In a rigid monoidal category \C\ the notions of \FBs\ and of \FCs\ on an
algebra $(A,m,\eta)$ are equivalent.
\\[3pt]
More specifically, for any algebra $A$ in \C\ the following holds:
\\[4pt]
{\rm(i)}\,~~There exists a \nondeg\ pairing on $A$ iff $A$ is isomorphic to $\va$
as an object of \,\C.
\\[4pt]
{\rm(ii)}\,~There exists an invariant pairing on $A$ iff there exists a
morphism from $A$ to $\va$ that is a morphism of left $A$-modules.
\end{prop}

\begin{proof}
Given a morphism $\varphi\iN\Hom(A,\va)$, we can define a pairing
$\kap_\varphi$ on $A$ by 
  \be
  \kap_\varphi := \tilde d_A\cir(\ida\oti\varphi) \,.
  \labl{defkapvarphi}
Conversely, given a pairing $\kap$ on $A$ we can define a morphism $\psi
\iN\Hom(A,\va)$ by $\psi\,{:=}\,\ikl$ as in \erf{pic_csp_03}.
Obviously the two operations are inverse to each other, in the sense that
  \be
  \Phi_{\kap_\varphi,\rm l} = \varphi \,.
  \ee
(i)~~Thus, given an {\em iso\/}morphism $\varphi\iN\Hom(A,\va)$, 
the morphism $\Phi_{\kap_\varphi,\rm l}$ associated to the pairing $\kap_\varphi$
is invertible, and hence $\kap_\varphi$ is \nondeg.
Conversely, given a {\em \nondeg\/} pairing $\kap$ on $A$, the morphism 
$\varphi\,{:=}\,\ikl\iN\Hom(A,\va)$ is an isomorphism.
\\[3pt]
(ii)\,~If a pairing on $A$ is of the form $\kap\eq\kap_\varphi$, then the statement
that it is invariant is precisely the equality \erf{pic_csp_11}. 
Hence if $\varphi\,{\equiv}\,\varphi_\rho$ is a {\em left module morphism\/},
then $\kap_\varphi$ is invariant.
Conversely, given an {\em invariant\/} pairing $\kap$ on $A$, the morphism
$\psi\,{:=}\,\ikl$ satisfies \erf{pic_csp_11}, and hence it is a
morphism of left $A$-modules.
\\[3pt]
Together, the validity of (i) and (ii) implies that if $(A,m,\eta,\irl)$ is an
algebra with \FCs, then the pairing $\kap_{\irl}$ is a \FBs\ on $A$. And 
conversely, if $(A,m,\eta,\kap)$ is an algebra with \FBs, then $\psi\eq\ikl$ 
is invertible, because $\kap$ is \nondeg, and hence it is a \FCs\ on $A$.
\end{proof}

We are now in a position to state

\begin{Def}\label{FRO}
A {\em \FRO\ algebra\/} in a rigid monoidal category \C\ is an algebra in \C\ for which
the following three equivalent conditions are satisfied:
\\[3pt]
(i)~~~There exists a \FAs\ on $A$.
\\[3pt]
(ii)~~There exists a \FBs\ on $A$.
\\[3pt]
(iii)~\,There exists a \FCs\ on $A$.
\end{Def}


\section{Symmetric Frobenius algebras}

The classical notion of symmetric Frobenius algebra can be formulated in the
categorical setting if the monoidal category \C\ is {\em sovereign\/}. This means 
that \C\ is rigid and that the left- and right-duality endofunctors are equal: 
That is, one has ${}^{\vee\!}U \eq U^\vee$ for all objects $U$, as well as 
${}^{\vee\!\!}f\eq f^\vee$, i.e.
  \be
  \bearl
  (\id_{{}^\vee U}\oti\tilde d_V) \circ (\id_{{}^\vee U}\oti f\oti\id_{{}^\vee V})
  \circ (\tilde b_U\oti\id_{{}^\vee V})
  \\{}\\[-.8em]\hspace*{7.5em}
  = (d_V\oti\id_{U^\vee}) \circ (\id_{V^\vee}\oti f\oti\id_{U^\vee}) \circ 
  (\id_{V^\vee}\oti b_U) ~\iN \Hom(V^\vee,U^\vee) \,,
  \eear
  \labl{vf=fv}
for all objects $U,V$ and all morphisms $f\iN\Hom(U,V)$.

\begin{rem}
As one particular aspect of sovereignty, we note that when applied to the left 
and right (co)evaluations, the equality \erf{vf=fv} amounts to
  \eqpic{pic_csp_20} {390} {28} {
\put(0,-1){
  \put(0,0)     {\IncludepicfuSt{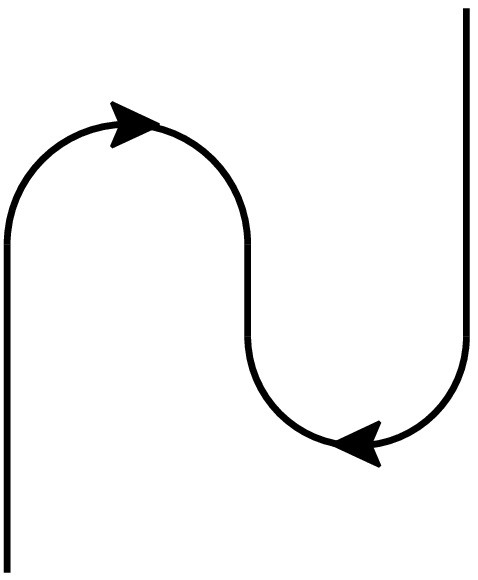}}
  \put(71,28)   {$ = $}
  \put(98,0)    {\IncludepicfuSt{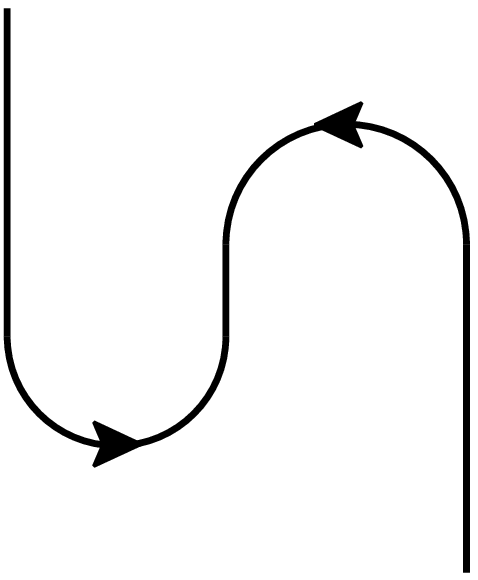}}
  \put(188,28)  {and}  
  \put(246,0)   {\IncludepicfuSt{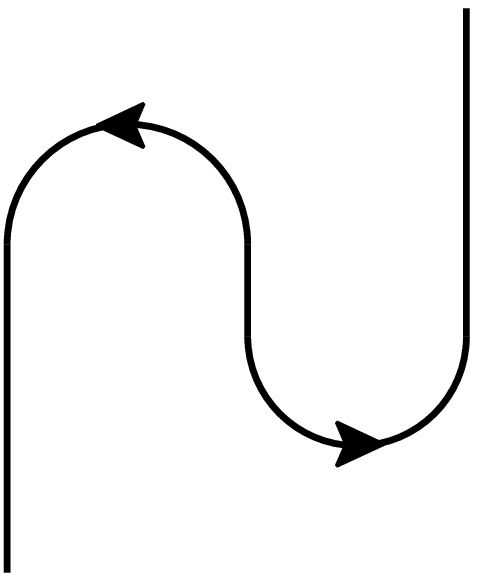}}
  \put(317,28)  {$ = $}
  \put(344,0)   {\IncludepicfuSt{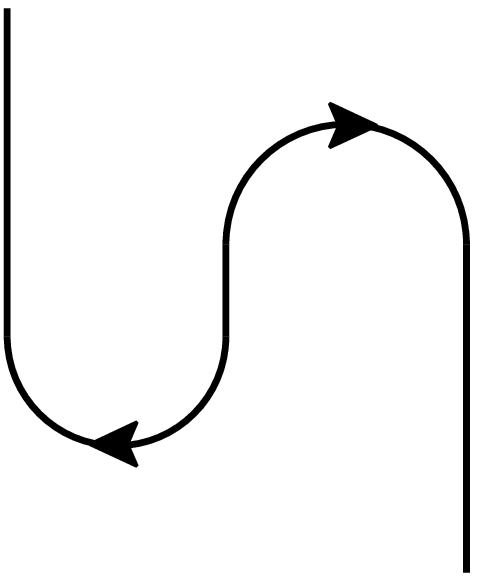}}
 \put(246,0){
  \put(48,68)     {\sse$U$}
  \put(-7,-10)    {\sse$^{\vee\!}U^\vee$}
  \put(95.5,68)   {\sse$U$}
  \put(141.5,-10) {\sse$^{\vee\!}U^\vee$}
 }
 \put(-245,0){
  \put(289.5,68)  {\sse$^{\vee\!}U^\vee$}
  \put(242.9,-10) {\sse$U$}
  \put(336.5,68)  {\sse$^{\vee\!}U^\vee$}
  \put(391.5,-10) {\sse$U$}
 }
} }
These equalities may be written as $\wee g\eq\eeW g$ and as
$(\eeW g)^{\!\vee}\eq{}^{\vee\!}(\wee g)$ for $g\eq\id_{{}^{\vee\!}U}\,{\equiv}\,
\id_{U^\vee}$, respectively, where for $f\iN\Hom(U,{}^{\vee\!}V)$ the morphism 
$\wee f$ is the one obtained from $f$ according to the prescription
\erf{wee}, while the morphism $\eEW f\iN\Hom(V,{}^{\vee\!}U)$ is defined by
  \be
  \eEW f := (\id_{{}^{\vee\!}U} \oti d_V) \circ
  (\id_{{}^{\vee\!}U} \oti f \oti \id_V) \circ (\tilde b_U \oti \id_V)
  \ee
for any $f\iN\Hom(U,V^{\vee})$.
\\
With the help of \erf{pic_csp_20} it is easy to check that, for ${}^{\vee\!}V$ 
equal to $V^\vee$, the sovereignty relation \erf{vf=fv} is equivalent to having
  \be
  \eEW f = \wee f
  \labl{wf=fw}
for all $f\iN\Hom(U,V^\vee)$.
Also note that the operations $\wee?$ and $\eew?$ are defined only on morphisms,
but not on objects, and that $\eew{\wee?}$ is the identity mapping.
\end{rem}

The formulation \erf{wf=fw} of the sovereignty relation is often quite
convenient. As an illustration, it allows one to quickly obtain the
following `opposite' version of the second identity in \erf{pic_csp_05}:
  \eqpic{pic_csp_21} {240} {30} {
\put(0,0){
  \put(0,0)      {\IncludepicfuSt{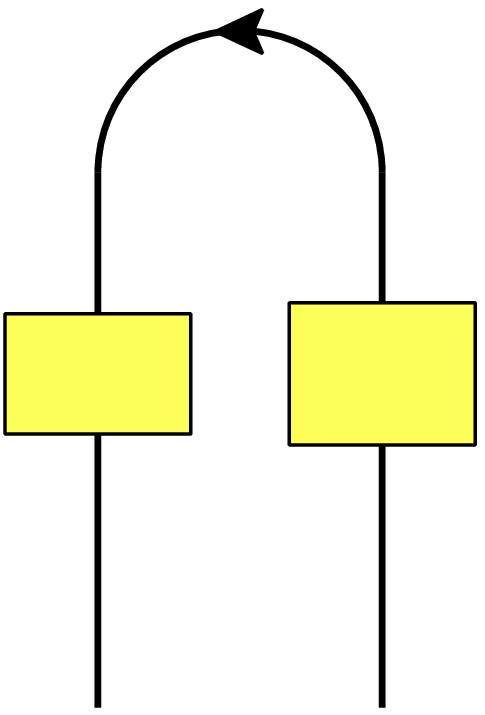}}
  \put(73,34)    {$ = $}
  \put(102,0)    {\IncludepicfuSt{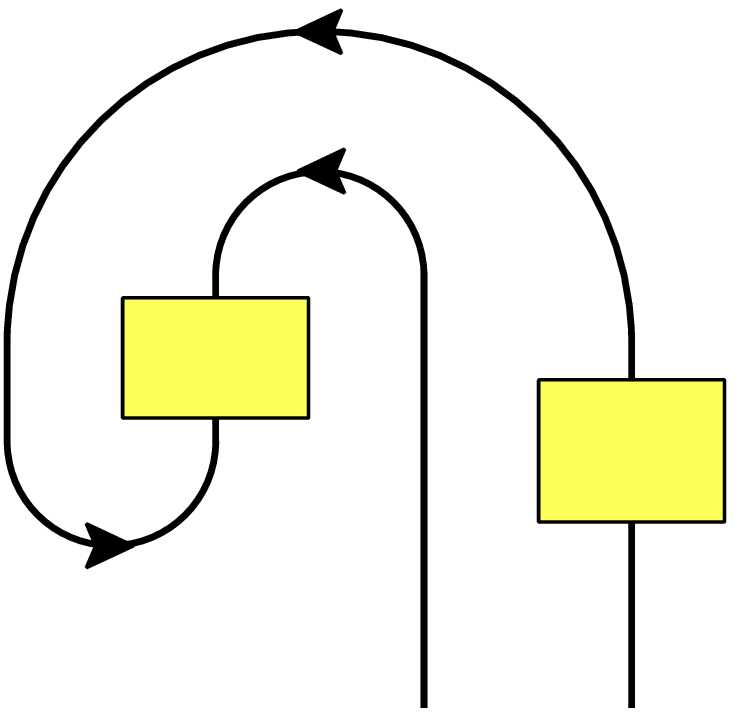}}
  \put(202,34)   {$ = $}
  \put(230,0)    {\IncludepicfuSt{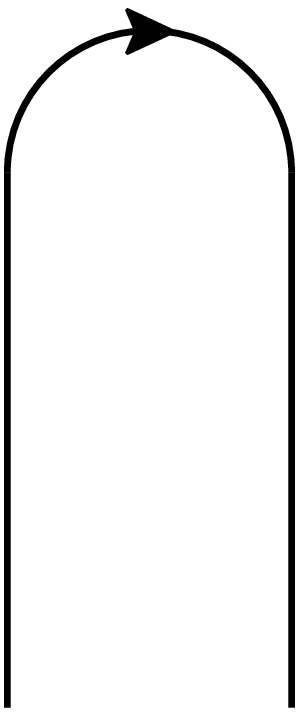}}
  \put(7,-10)     {\sse$A$}
  \put(36,-10)    {\sse$\av$}
  \put(1,34.2)    {$\ikl$}
  \put(32,34.2)   {$\ikrm$}
  \put(145,-10)   {\sse$A$}
  \put(166.1,-10) {\sse$\av$}
  \put(115.5,35.8){$\ikr$}
  \put(161.2,25.6){$\ikrm$}
  \put(227,-10)   {\sse$A$}
  \put(258,-10)   {\sse$\va$}
} }

\begin{rem} \label{rem2}
As a matter of fact, for introducing the notion of symmetric \Fa\ it is already 
sufficient that $\va$ and $\av$ are equal as objects of \C. Similarly, 
all the results below remain true if one has $\va\eq\av$ as well as 
${}^{\vee\!\!}f\eq f^\vee$ for just a few particular morphisms $f$, like for the
product $m$ of $A$ and the morphisms $\ikl\iN\Hom(A,\va)$ defined in 
\erf{pic_csp_03}. However, the only situation known to us in which these 
weaker conditions are satisfied naturally is that \C\ is indeed sovereign.
(Compare the analogous comments on rigidity in Remark \ref{rem1}.)
\end{rem}

We start with the notion of symmetric \Fa\ that is used e.g.\ in \cite{fuRs,muge8}.

\begin{Def}\label{FAs}
A {\em symmetric\/} \FAs\ on an algebra $A\eq(A,m,\eta)$ in a sovereign monoidal 
category \C\ is a \FAs\ $(\Delta,\eps)$ for which the endomorphism 
  \be
  \nake := (d_A\oti\ida) \cir [\idav \oti (\Delta\cir\eta\cir\eps\cir m)]
  \cir (\tilde b_A\oti\ida) 
  \labl{nakedef}
of $A$ equals $\ida$. 
\end{Def}

Note that by the Frobenius relations \erf{picfrob} the equality $\nake\eq\ida$ 
is equivalent to the equality $\phil\eq\phir$ between the two morphisms from $A$
to $\va\eq\av$ that we introduced in \erf{defphilphir} and which, owing to the 
Frobenius relations, are isomorphisms, and are related to the morphism
\erf{nakedef} by $\nake\eq\phir^{-1}\cir\phil$. Pictorially this identity reads
  \eqpic{pic_csp_13} {135} {33} {
\put(0,-1){
  \put(-55,34)  {$ \phil \,= $}
  \put(0,0)     {\IncludepicfuSt{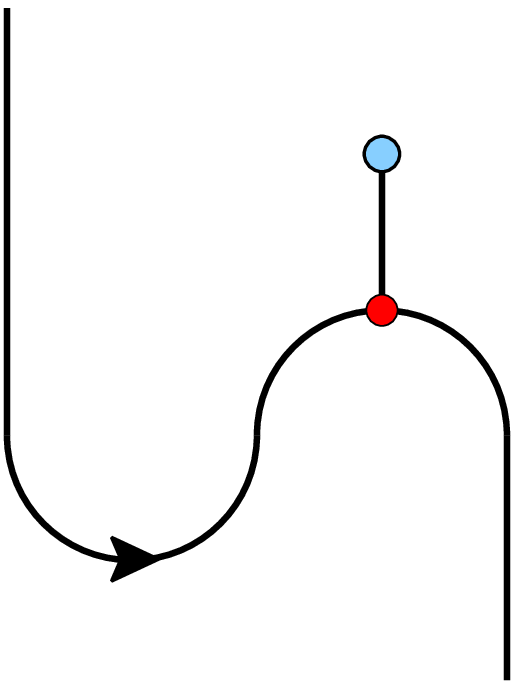}}
  \put(72,34)   {$ = $}
  \put(100,0)   {\IncludepicfuSt{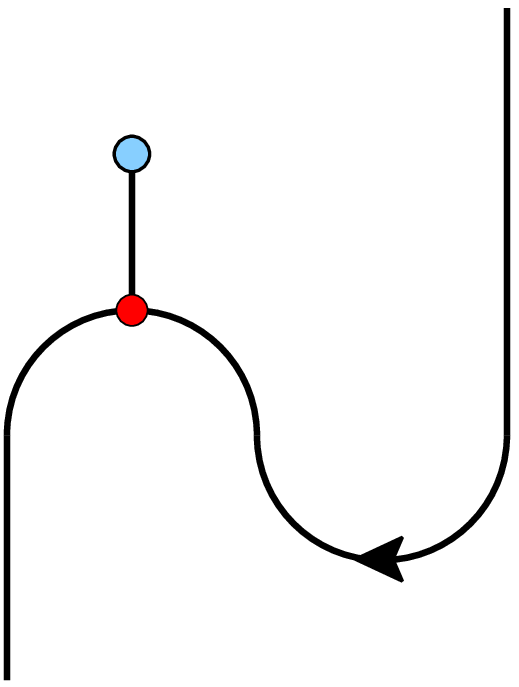}}
  \put(175,34)  {$ =\, \phir \,. $}
  \put(-3,79.2) {\sse$\va$}
  \put(51.6,-8.8){\sse$A$}
  \put(151,79.2){\sse$\av$}
  \put(96.7,-8.8){\sse$A$}
} }

\begin{Def}\label{FBs}
A {\em symmetric\/} \FBs\ on an algebra $A\eq(A,m,\eta)$ in a sovereign monoidal 
category \C\ is a \FBs\ $\kap\iN\Hom(A,\oti A,\one)$ on $A$ that is symmetric in 
the sense that the equality
  \be
  d_A \circ (\idva\oti\kap\oti\ida) \circ (\tilde b_A\oti\ida\oti\ida) = \kap
  \labl{eqFBs}
holds.
\end{Def}

Note that on the \lhs\ of \erf{eqFBs}, the evaluation 
$d_A\iN\Hom(\av{\otimes}\,A,\one)$ 
is composed with a morphism in $\Hom(A\oti A,\va\oti A)$, which in the present 
context is the reason why we need $\va\eq\av$. Pictorially, \erf{eqFBs} reads
  \eqpic{pic_csp_14} {110} {21} {
\put(0,-2){
  \put(0,0)     {\IncludepicfuSt{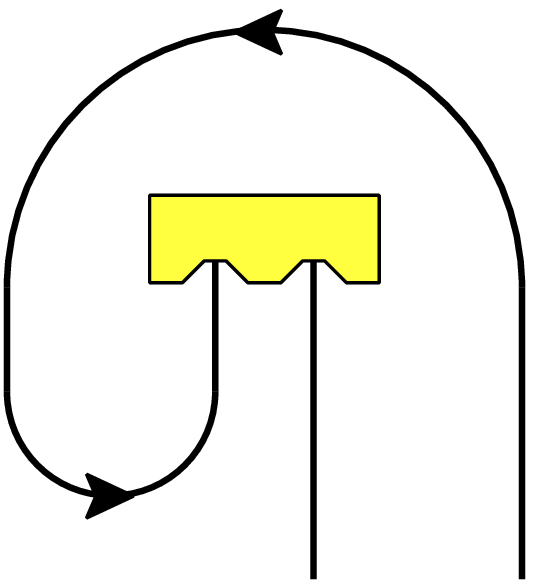}}
  \put(76,23)   {$ = $}
  \put(104,0)   {\IncludepicfuSt{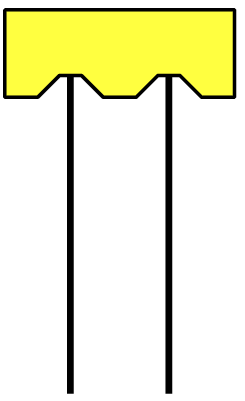}}
  \put(30.4,-9.2)  {\sse$A$}
  \put(53.4,-9.2)  {\sse$A$}
  \put(107.7,-9.2) {\sse$A$}
  \put(118.5,-9.2) {\sse$A$}
} }
By the defining properties of the (co)evaluation, this relation is equivalent to
  \eqpic{pic_csp_15} {110} {21} {
\put(0,-2){
  \put(0,0)     {\IncludepicfuSt{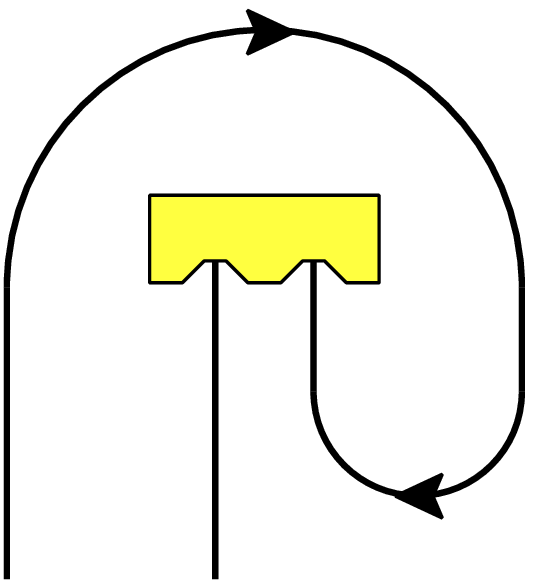}}
  \put(76,23)   {$ = $}
  \put(104,0)   {\IncludepicfuSt{14b}}
  \put(-3.3,-9.2)  {\sse$A$}
  \put(19.4,-9.2)  {\sse$A$}
  \put(107.7,-9.2) {\sse$A$}
  \put(118.5,-9.2) {\sse$A$}
} }

Next note that when $A$ has a \FCs\ and the equality $\va\eq\av$ between
left and right dual objects holds, then $A$ is isomorphic
to $\va\eq\av$ both as a left and as a right module.
Accordingly the following definition is natural.

\begin{Def}\label{FCs}
A {\em symmetric\/} \FCs\ on an algebra $A\eq(A,m,\eta)$ in a sovereign 
monoidal category \C\ is an isomorphism $\irl$ from $A$ to $\va\eq\av$ that is
both a morphism of left $A$-modules and a morphism of right $A$-modules (and thus
a morphism of $A$-bimodules).
\end{Def}

In terms of the notations used in the proof of Lemma \ref{lem1}, that a \FCs\
is symmetric means that we have $\irr\eq\irl$, or equivalently, that
the isomorphism $\irl$ satisfies
  \eqpic{pic_csp_16} {120} {25} {
\put(0,3){
  \put(0,0)     {\IncludepicfuSt{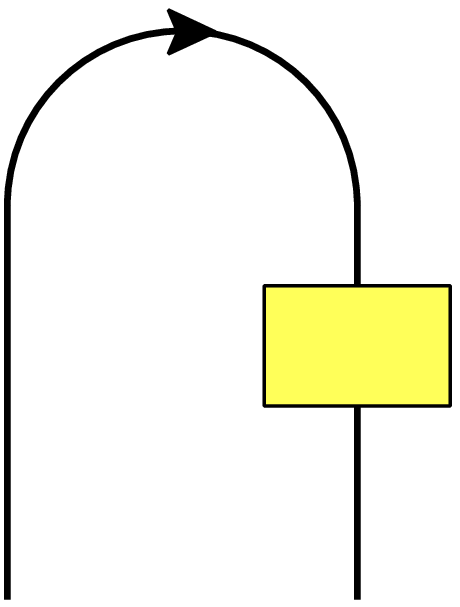}}
  \put(64,23)   {$ = $}
  \put(92,0)    {\IncludepicfuSt{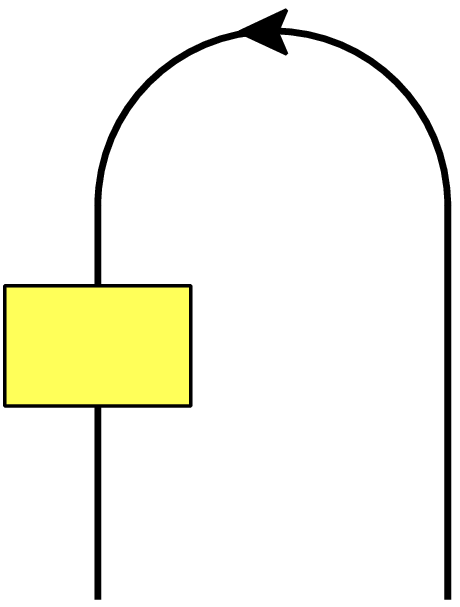}}
  \put(-3.5,-10) {\sse$A$}
  \put(34.9,-10) {\sse$A$}
  \put(32.7,25.4){$\irl$}
  \put(98.4,-10) {\sse$A$}
  \put(136.8,-10){\sse$A$}
  \put(96.2,25.4){$\irl$}
} }

\begin{prop}
In a sovereign monoidal category \C\ the notions of symmetric \FAs, 
symmetric \FBs\ and symmetric \FCs\ are equivalent.
\end{prop}

\begin{proof}
Recalling the relations between the characteristic morphisms of \FAs s,
\FBs s and \FCs s the assertion is close to a tautology. Let us nonetheless
write out the proof.
\\[3pt]
(i)~~Given a symmetric \FAs, define the pairing $\kap_\eps$ as in \erf{kapepsdef}.
For $\kap\eq\kap_\eps$, the symmetry property \erf{pic_csp_14} is satisfied
because it is nothing but (up to composition with appropriate (co)evaluations)
the equality \erf{pic_csp_13}. Conversely, given a symmetric \FBs, define the
counit $\eps_\kap$ as in \erf{deltakapepskapdef}. Then the symmetry property 
\erf{pic_csp_13} is satisfied because owing to $\eps_\kap\cir m \eq \kap$
it is nothing but (up to composition with appropriate (co)evaluations) the 
equality \erf{pic_csp_14}.
\\
Thus the notions of symmetric \FAs\ and of symmetric \FBs\ are equivalent.
\\[3pt]
(ii)\,~Given a symmetric \FCs, define the pairing $\kap_\irl$ as in 
\erf{defkapvarphi} with $\varphi\eq\irl$. For $\kap\eq\kap_\irl$, the symmetry 
property \erf{pic_csp_14} is satisfied because after use of the defining property
of $\tilde b_A$ and $\tilde d_A$ it is nothing but the equality \erf{pic_csp_16}.
Conversely, given a symmetric \FBs, define the morphism $\irl\,{:=}\,\ikl$
as in \erf{pic_csp_03}. Then the symmetry property \erf{pic_csp_16}
is satisfied because, again after use of the duality axioms, it is nothing but 
the equality \erf{pic_csp_14}.
\\
Thus the notions of symmetric \FCs\ and of symmetric \FBs\ are equivalent.
\end{proof}

In analogy with Definition \ref{FRO} we now give

\begin{Def}\label{sFRO}
A {\em symmetric \Fa\/} in a sovereign monoidal category \C\ is an algebra in 
\C\ for which the following three equivalent conditions are satisfied:
\\[3pt]
(i)~~~There exists a symmetric \FAs\ on $A$.
\\[3pt]
(ii)~~There exists a symmetric \FBs\ on $A$.
\\[3pt]
(iii)~\,There exists a symmetric \FCs\ on $A$.
\end{Def}

It is worth pointing out that none of these structures requires \C\ to be 
braided. If \C\ does have a braiding, then it can of course be used to
reformulate the property of $\eps$, $\kap$ and $\irl$ of being symmetric in
a manner that resembles more closely the customary description in the vector
space case.


\section{Nakayama automorphisms}

Given an algebra $A$ in a sovereign monoidal category \C\ and a pairing $\kap$ 
on $A$, we call an endomorphism $\naka\,{\equiv}\,\naka_\kap$ a {\em Nakayama
morphism \/} iff $d_A \cir (\idav\oti\kap\oti\ida) \cir (\tilde b_A\oti\ida\oti
\ida) \eq \kap \cir (\naka\oti\ida)$ (see e.g.\ \cite[\S16E]{LAm} for the 
classical case). Pictorially, the defining relation for $\naka$ reads
  \eqpic{pic_csp_17} {120} {22} {
\put(0,0){
  \put(0,0)     {\IncludepicfuSt{14a}}
  \put(76,23)   {$ = $}
  \put(104,0)   {\IncludepicfuSt{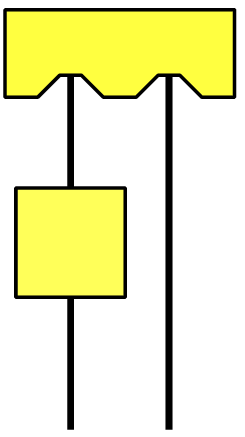}}
  \put(30.2,-10)  {\sse$A$}
  \put(53.1,-10)  {\sse$A$}
  \put(107.5,-10) {\sse$A$}
  \put(118.3,-10) {\sse$A$}
  \put(107.5,17.2){$\naka$}
} }
Defining  $\ikl\iN\Hom(A,\va)$ and $\ikr\iN\Hom(A,\av)$ as in \erf{pic_csp_03} and
\erf{pic_csp_04}, one can rewrite this relation as
  $ \ikr \cir \naka \eq \ikl $.
Now if $A$ is a Frobenius algebra with \FBs\ $\kap$, then $\ikr$ is
invertible, so that
  \be
  \naka = \ikrm \circ \ikl \,,
  \labl{naka=ikrm ikl}
and in particular $\naka$ is an \auto\ of $A$ as an object of \C.
(Note that the \rhs\ of \erf{naka=ikrm ikl} may also be written as
$\ikrm\cir\eew{(\ikr)}$; $\naka$ belongs therefore to the class of automorphisms
$\mathcal V_{\!A}$ that are used in \cite{fuSc16}
for assigning a Frobenius-Schur indicator to $A$.)

\begin{prop}\label{Nakaalgmor}
Any \Naka\, $\naka$ of a \Fa\ $A$ in a sovereign monoidal category 
is a unital algebra morphism.
\end{prop}

\begin{proof}
Any morphism $\omega$ that is an \auto\ of $A$ as an associative algebra
is automatically also unital, i.e.\ satisfies $\omega\cir\eta\eq\eta$. 
\\
It is therefore sufficient to show that $\naka$ is compatible with the 
product of $A$. We demonstrate this property by showing that 
$\naka^{-1}\!\cir m\cir(\naka\oti\ida) \eq m\cir(\ida\oti\naka^{-1})$. 
Consider the following chain of equalities:
  \eqpid{pic_csp_22} {400} {45} {
\put(-40,-9){
  \put(2,48)     {$\naka^{-1}\!\circ m\circ (\naka\oti\ida) ~= $}
  \put(132,0)    {\IncludepicfuSt{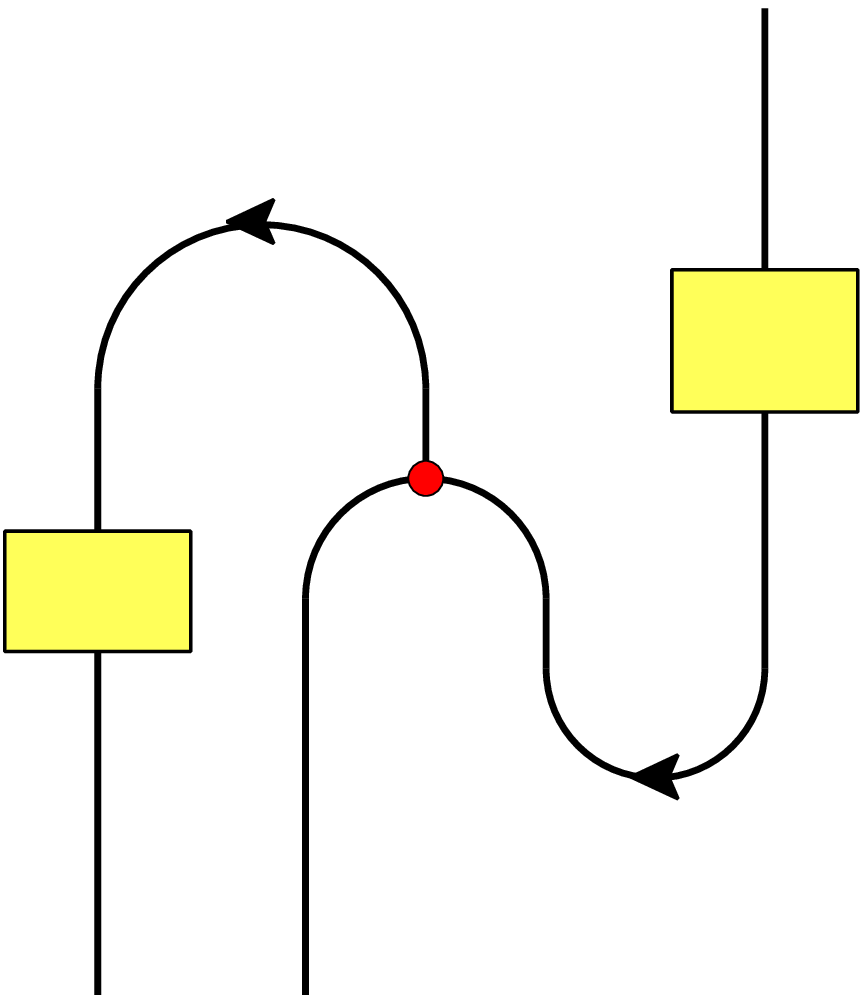}}
  \put(239,48)   {$ = $}
  \put(265,0)    {\IncludepicfuSt{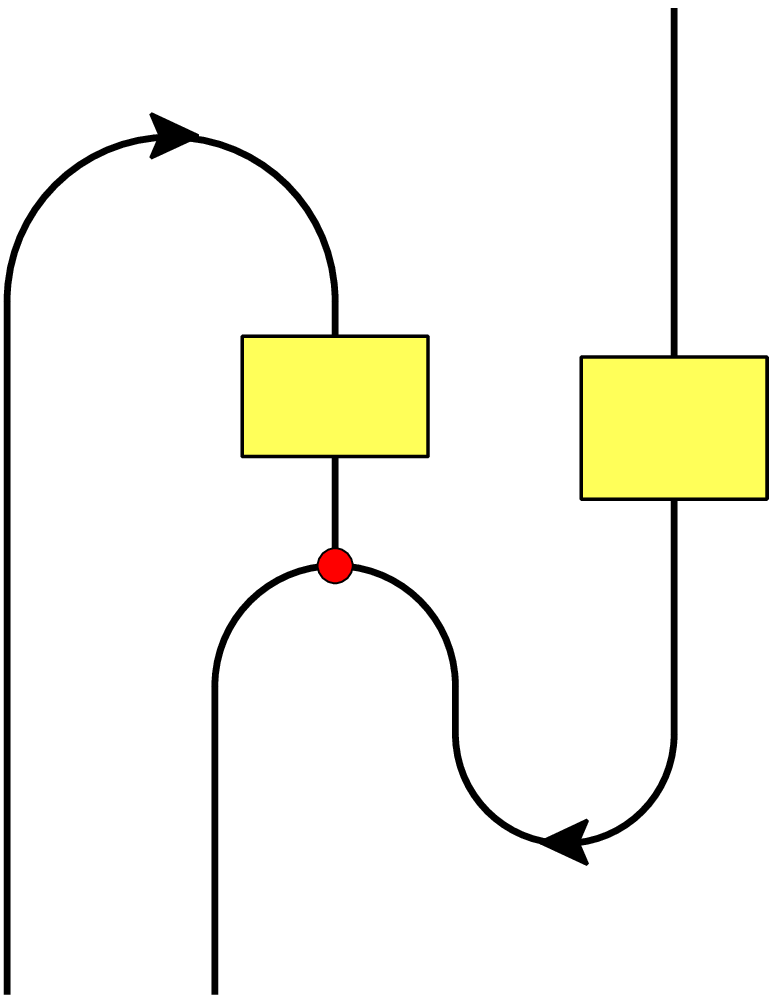}}
  \put(360,48)   {$ = $}
  \put(385,0)    {\IncludepicfuSt{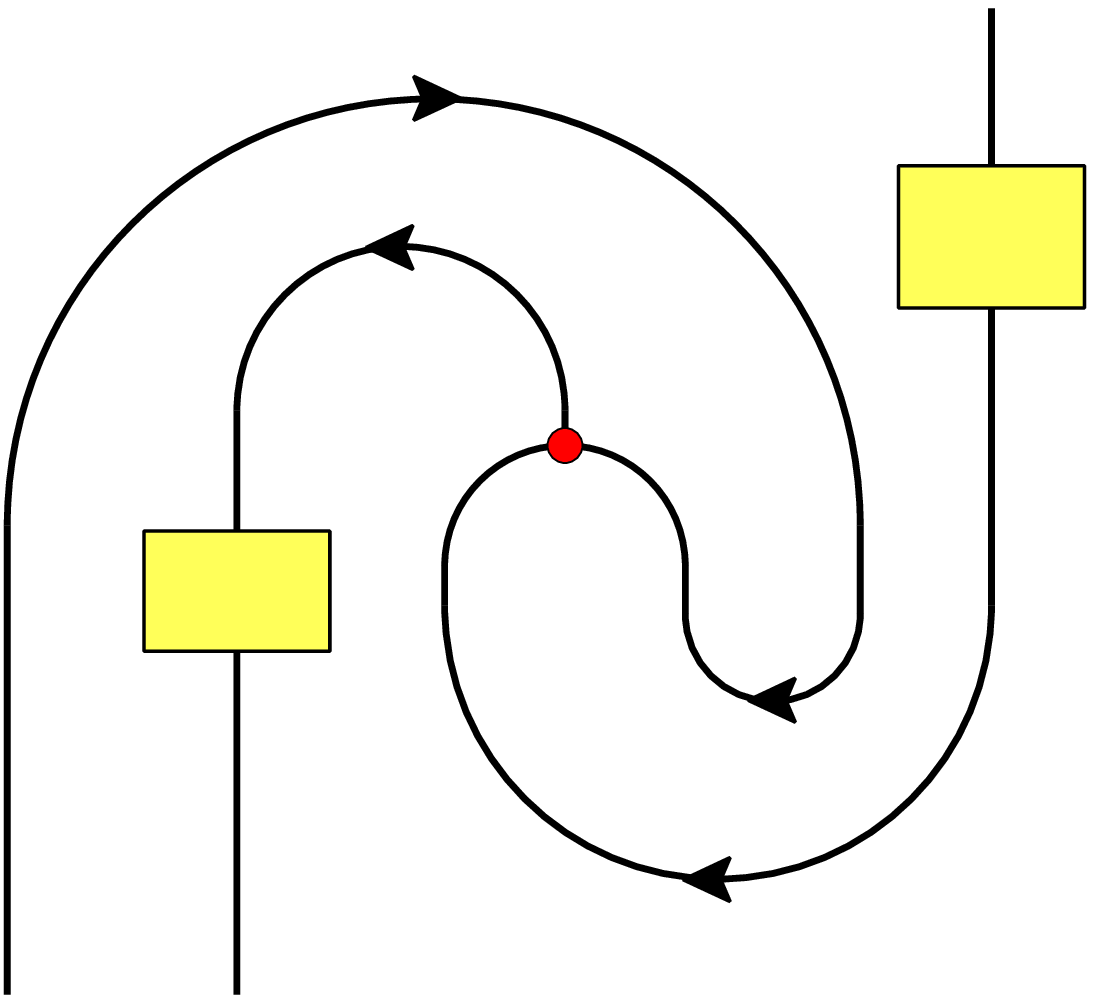}}
  \put(139,-10)    {\sse$A$}
  \put(161.8,-10)  {\sse$A$}
  \put(132.5,41.8) {$\ikl$}
  \put(205.9,69.7) {$\iklm$}
  \put(212,114)    {\sse$A$}
  \put(262,-10)    {\sse$A$}
  \put(284.8,-10)  {\sse$A$}
  \put(292,63.5)   {$\ikr$}
  \put(329,60)     {$\iklm$}
  \put(335,114)    {\sse$A$}
  \put(382,-10)    {\sse$A$}
  \put(407.2,-10)  {\sse$A$}
  \put(401,42)     {$\ikr$}
  \put(484,81)     {$\iklm$}
  \put(490,114)    {\sse$A$}
} }
The first of these follows by combining \erf{naka=ikrm ikl} with the
fact that $\ikr$ is a morphism of right $A$-modules and using the explicit form
\erf{rhol,rhor} of the right action of $A$ on $\av$; the second equality is a
consequence of $\ikr\eq\wee{(\ikl)}$; and the third expresses again the fact 
that $\ikr$ is a morphism of right $A$-modules.
\\
Next we apply the sovereignty relation \erf{vf=fv} to the product that is 
contained in the morphism on the \rhs\ of \erf{pic_csp_22}, thereby 
obtaining the \lhs\ of the following equality:
  \eqpic{pic_csp_23} {400} {49} {
\put(-20,-2){
  \put(0,0)     {\IncludepicfuSt{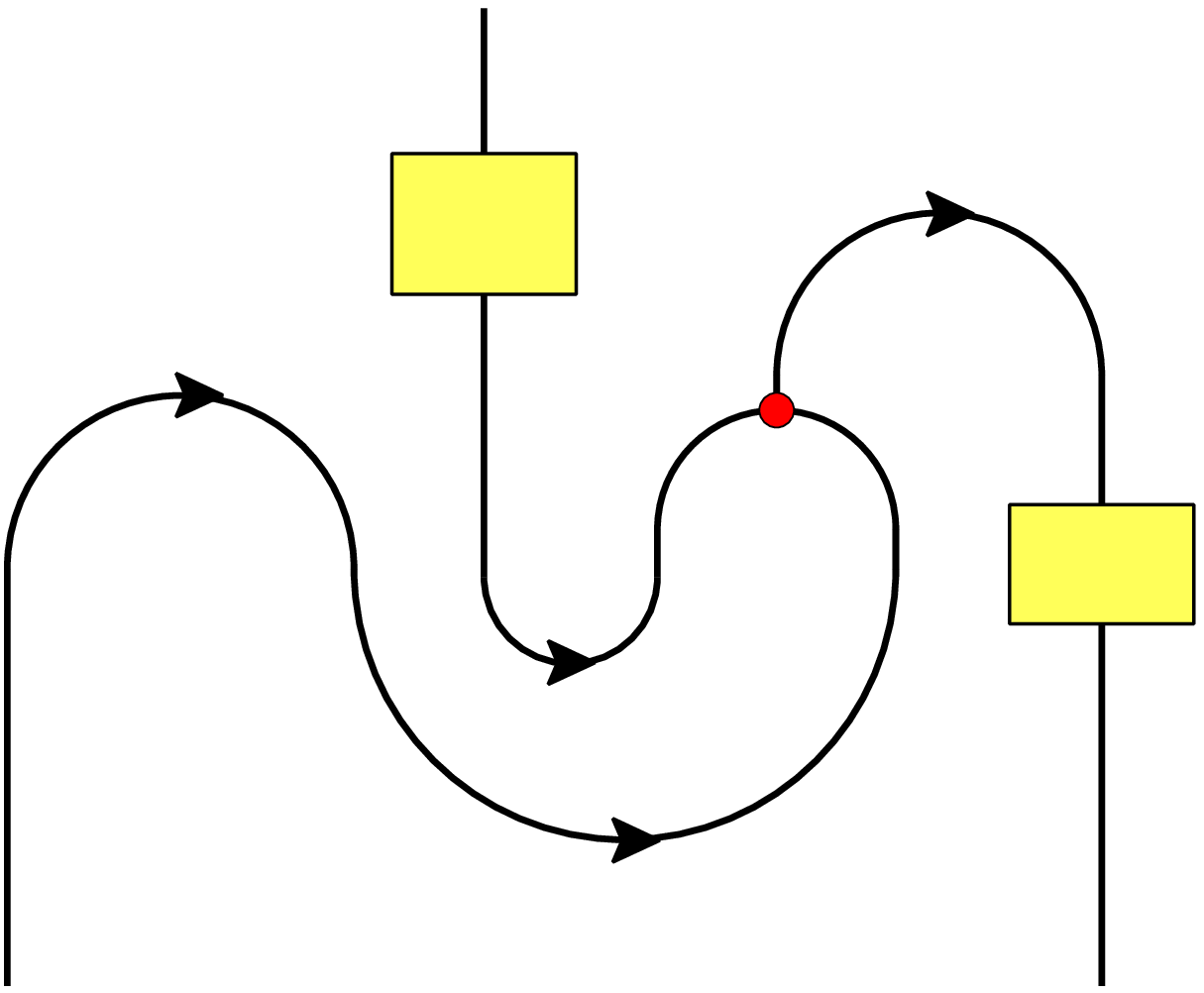}}
  \put(153,49)  {$ = $}
  \put(180,0)   {\IncludepicfuSt{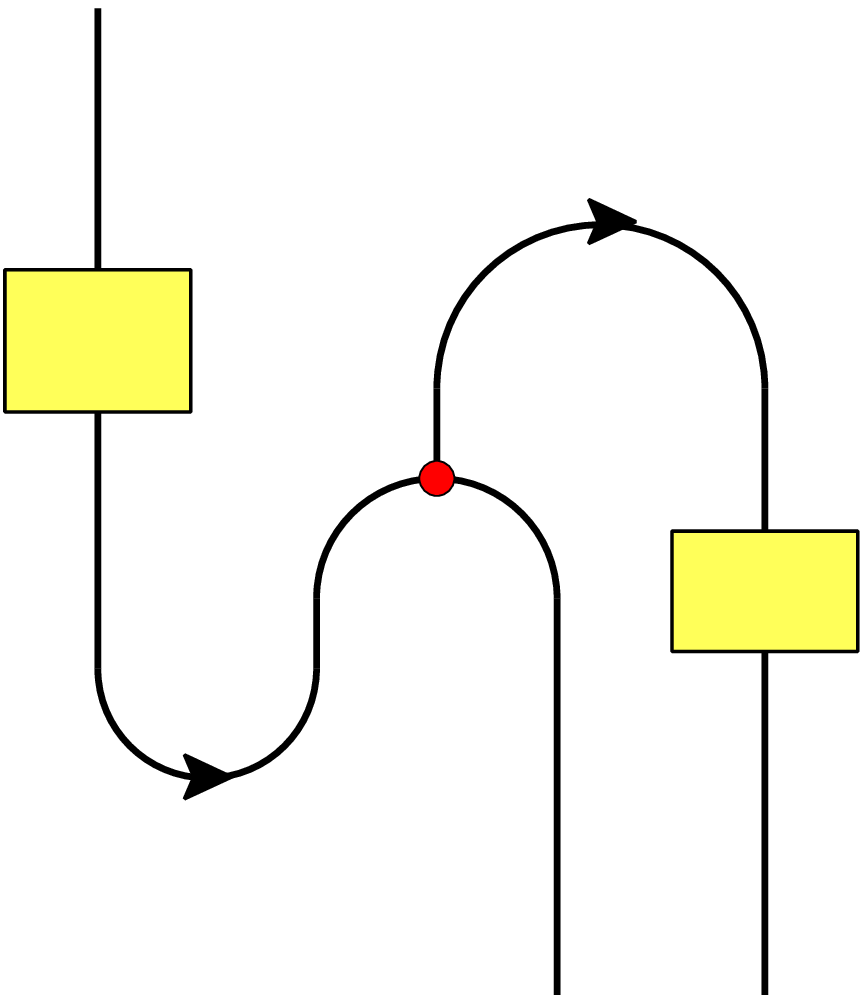}}
  \put(293,49)  {$ =~ \iklm \circ \rho \circ (\ida \otimes \ikr)\,. $}
  \put(-3,-10)     {\sse$A$}
  \put(49.5,114)   {\sse$A$}
  \put(118.5,-10)  {\sse$A$}
  \put(43.7,82.2)  {$\iklm$}
  \put(112,44.4)   {$\ikr$}
  \put(237,-10)    {\sse$A$}
  \put(260,-10)    {\sse$A$}
  \put(187,114)    {\sse$A$}
  \put(180.7,69.5) {$\iklm$}
  \put(253.9,42)   {$\ikr$}
} }
Here $\rho$ the left action of $A$ on $\va$ as in \erf{rhol,rhor}.
Owing to the left-module morphism property of $\iklm$,
the right hand side of \erf{pic_csp_23} is, in turn, equal to 
$m\cir(\ida\oti(\iklm{\circ}\,\ikr)) \eq m\cir(\ida\oti\naka^{-1})$.
\end{proof}

According to Definition \ref{sFRO}, a \Fa\ $A$ is symmetric iff for some choice 
of $\kap$ the morphism $\naka_{\kap}$ is the identity morphism. As we will see,
this also means that $A$ is symmetric iff every \Naka\ of $A$ is an inner \auto.
To derive this characterization, we need to introduce some further terminology.
First notice that the set $\ao\,{:=}\,\Hom(\one,A)$ has an associative product
$\mo{:}\ \ao\Times\ao\To\ao$ given by convolution,
  $ \mo(a,b) \,{:=}\, m \cir (a \oti b) $;
the unit $\eta$ of $A$ acts as an identity element, 
$\mo(a,\eta)\eq a\eq \mo(\eta,a)$. Owing to the associativity 
of $m$ the algebra $A$ carries the structure of a bimodule over $\ao$, with 
the left and right action of $a\iN\ao$ given by
  \be
  \lo a := m\cir(a\oti\ida) \qquad{\rm and}\qquad
  \ro a := m\cir(\ida\oti a) \,,
  \ee
respectively. Note that the assignments $\lo{},\ro{}\,{:}~\ao\To\End(A)$ mapping
$a\iN\ao$ to $\lo a$ and $\ro a$, respectively, are the analogues of the left and 
right regular \rep\ of an ordinary algebra in a category of vector spaces -- for 
every $a\iN\ao$ the action $\lo a$ furnishes an endomorphism of $A$ as a right 
$A$-module, while $\ro a$ furnishes an endomorphism of $A$ as a left $A$-module.

Of particular interest is the subset of $\ao$ consisting of morphisms
that are invertible with respect to the product $\mo$. These form a group,
the group of units of $A$, which we will denote by $\ua$.
For $g\iN\ua$, $\lo g$ and $\ro g$ are \auto s of $A$ (as an object, and
also as a right and left $A$-module, respectively), with inverses $\lom g$ and
$\rom g$, respectively. Moreover, the composition
  \be
  \ad g := \lom g \circ \ro g = \ro g \circ \lom g
  \ee
with $g\iN\ua$ is an \auto\ of $A$ as an algebra. An automorphism of $A$ is
called an {\em inner \auto\/} iff it is of this particular form; clearly, the
inner \auto s form a group under composition, with 
$\ad g\cir\ad h\eq\ad{\mo(h,g)}$ and $\ad g^{\,-1}\eq\ad{g^{-1}}$.\,%
 \footnote{~Inner \auto s can be used to twist the left or right action of $A$
 on itself, and thereby play a prominent role in the study of the category of
 $A$-bimodules, e.g.\ (for algebras in abelian monoidal categories) for the 
 description of the Picard group of the bimodule category and its appearance
 in a Rosenberg-Zelinsky exact sequence \cite{vazh,fuRs11,bfrs}.}

The module morphism properties of $\ro a$ and $\lo a$ imply that if $\kap$ 
is an invariant pairing on $A$, then so are $\kap\cir(\ida\oti\ro a)$ and
$\kap\cir(\lo a\oti\ida)$ for any $a\iN\ao$. Moreover,
for $g\iN\ua$ the composition of any \auto\ of $A$ with $\ro g$ or $\lo g$ is
again an \auto; thus if $\kap$ is a \nondeg\ pairing, then so are 
$\kap\cir(\ida\oti\ro g)$ and $\kap\cir(\lo g\oti\ida)$. 
Conversely, we have

\begin{lemma}\label{lemkapkapp}
Any two invariant \nondeg\ pairings $\kap$ and $\kap'$ on an algebra $A$ 
in a rigid monoidal category differ by composition with an endomorphism 
of the form $\id_A\oti\ro g$ for some $g\iN\ua$.
\end{lemma}

\begin{proof}
If the pairings $\kap$ and $\kap'$ are \nondeg, then the corresponding 
morphisms $\ikl$ and $\ikpl$ in $\Hom(A,\va)$ are isomorphisms, and hence
$\ikpl\eq\ikl\cir\sigl$ for some \auto\ $\sigl$ of $A$.
Equivalently, $\kap$ and $\kap'$ are related by 
  \be
  \kap' = \kap\cir(\id_A\oti\sigl) \,.
  \ee
$\sigl$ is even an \auto\ of $A$ as a left $A$-module, because $\ikl$ and 
$\ikpl$ are. Rewriting $\sigl$ identically as $\sigl\cir m\cir(\ida\oti\eta)$ 
it follows in particular that 
  \be 
  \sigl = m \circ [\ida \oti (\sigl\cir\eta)] = \ro{\sigl\circ\eta} \,.
  \ee
Thus indeed $\kap' \eq \kap\cir (\id_A\oti\ro g)$, with $g\eq\sigl\cir\eta\,$.
\\
Finally, $g$ is invertible, with inverse $g^{-1}\eq\sigL^{-1}\cir\eta$.
\end{proof}

Note that only `left-nondegeneracy' of the pairings $\kap$ and $\kap'$
(and only left-rigidity of \C) enters the proof.
An analogous argument based on right-nondegeneracy (and right-rigidity)
shows that the two invariant \nondeg\ pairings are also related by 
$\kap'\eq\kap\cir(\sigr\oti\ida)$ with $\sigr\eq\lo{\sigr\circ\eta}$ for
some invertible right $A$-module morphism $\sigr\iN\End(A)$.

As a consequence of Lemma \ref{lemkapkapp} we have

\begin{prop}
Any two \Naka s of a \Fa\ $A$ in a sovereign monoidal category \,\C\ 
differ by composition with an inner \auto\ of $A$.
\end{prop}

\begin{proof}
Recall from the proof of the lemma that
  \be
  \ikpl = \ikl \circ \ro g
  \labl{ikpl=ikl ro g}
for some $g\iN\ua$. Together with \erf{ikr=iklwee} it then follows that
$\ikpr \eq \wee{(\ikpl)} \eq \ro g^{\vee} \cir \wee{(\ikl)}$ and thus
  $ \ikprm  
  \eq \ikrm \cir \romv g $.
Now by using sovereignty together with the second identity in \erf{pic_csp_05}
one can rewrite the latter equality as
  \eqpic{pic_csp_18} {190} {44} {
\put(0,2){
  \put(0,47)    {$ \ikprm \,=\, \ikrm \circ {}^{\vee\!}\rom g \,= $}
  \put(130,0)   {\IncludepicfuSt{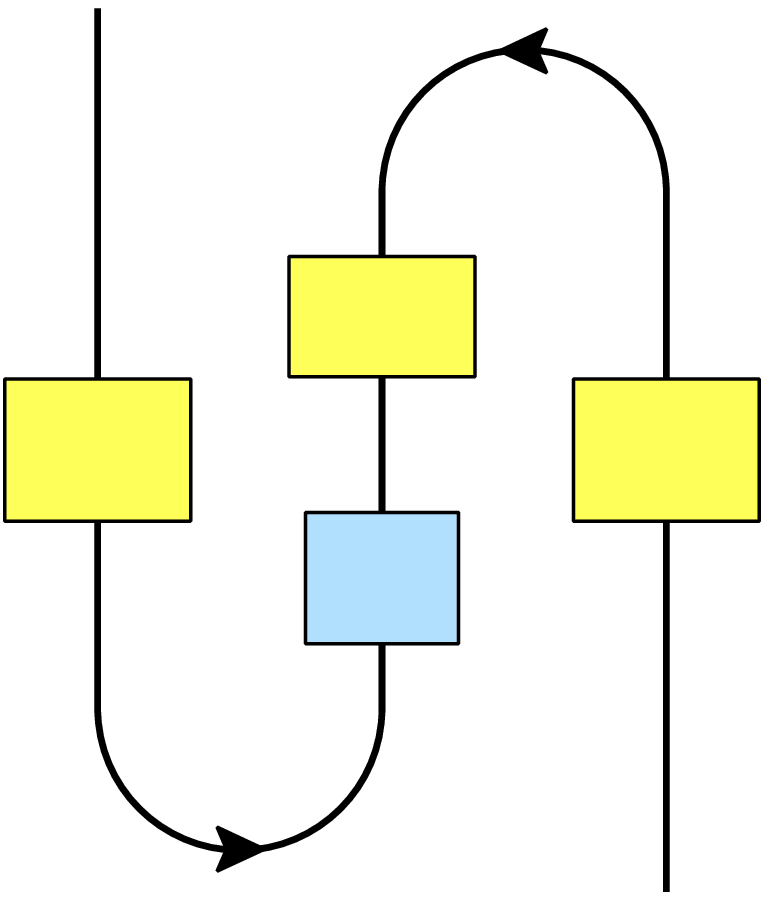}}
  \put(137.3,103)  {\sse$A$}
  \put(198.4,-10)  {\sse$\av$}
  \put(130.6,46.2) {$\ikrm$}
  \put(161.8,60.9) {$\ikr$}
  \put(193.5,46.2) {$\iklm$}
  \put(164.7,34.1) {$r_{\!g^{\!-\!1}}$}
} }
Furthermore, since $\ikr$ is a morphism of right $A$-modules, for any 
$h\iN\ao$ one has 
  \be
  \ikr \circ \ro h = (d_A\oti\idav) \circ (\ikr\oti\lo h\oti\idav)
  \circ (\ida\oti b_A) \,,
  \ee
and as a consequence we can rewrite \erf{pic_csp_18} as
  $ \ikprm \eq \ikrm \cir \ikl^{} \cir \lom g \cir \iklm $.
When combined with \erf{naka=ikrm ikl} and \erf{ikpl=ikl ro g}, we therefore
conclude that the \Naka s $\naka\,{\equiv}\,\naka_{\kap}$ and
$\naka'\,{\equiv}\,\naka_{\kap'}$ are related by
  \be
  \naka' = \ikprm \circ \ikpl = \ikrm \circ \romv g \circ \ikl \circ \ro g
  = \ikrm \circ \ikl \circ \lom g \circ \ro g = \naka \circ \ad g \,,
  \labl{nakad}
thus proving the claim.  
\end{proof}

It follows in particular that when selecting a \FAs\ $\eps$ on $A$ we can write
any \Naka\ $\naka$ of $A$ in the form $\naka\eq\naka_{\kap_\eps}{\circ}\,\ad g$
with some $g\iN\ua$. Thus combining the statements that $\naka_{\kap_\eps}$ is 
an algebra morphism (see the proposition in Section 4 of \cite{jf29}) and that 
any inner automorphism of $A$ is an algebra morphism as well (see above) 
provides an alternative derivation of Proposition \ref{Nakaalgmor}.
Also note that by combining \erf{nakad} with the results about the morphisms
$g\eq\sigl\cir\eta$ and $h\eq\sigr\cir\eta$ obtained above one finds that 
$\naka_\kap\cir\lo g \eq \lo h\cir\naka_\kap$ which, in turn, by the algebra
morphism property of $\naka_\kap$ implies that $h\eq\naka_\kap\cir g$.

Further, if there exists a \FBs\ $\kap$ on $A$ such that the associated \Naka\
$\naka_\kap$ is inner, say $\naka_\kap\eq\ad h$, then one has 
$\naka_{\kap'}\eq\ida$ for the \FBs\ $\kap'\,{:=}\,\kap\cir(\lom h\oti\ida)$ on 
$A$, and so $A$ is symmetric. Thus indeed a \Fa\ $A$ is symmetric iff every
\Naka\ of $A$ is inner.

\medskip

To conclude, we combine Lemma \ref{lemkapkapp}, and the remarks preceding
it, with the formulas \erf{pic_csp_03} and \erf{deltakapepskapdef} to arrive at

\begin{cor}
Let $A\eq(A,m,\eta)$ be an algebra in a rigid monoidal category.
\\[3pt]
{\rm(i)}~~Let $\eps,\kap,\irl$ be $\eps$-, $\kap$- and \FCs s on $A$,
respectively, and let $g,h\iN\ua$. Then also $\,\eps\cir\ro g\cir\lo h$,
$\kap\cir(\lo g\oti\ro h)$ and $\,\irl\cir\ro g$ 
are $\eps$-, $\kap$- and \FCs s on $A$, respectively.
\\[3pt]
{\rm(ii)}~\,Any two triples $(\eps,\kap,\irl)$ and $(\eps',\kap',\irl')$
of $\eps$-, $\kap$- and \FCs s on $A$ are related as in {\rm(i)} for some 
choice of $g,h\iN\ua$.
\end{cor}

\bigskip

\noindent{\sc Acknowledgments:}\\
We thank A.A.\ Davydov, V.\ Hinich, C.\ Schweigert and A.\ Stolin for helpful 
discussions, and L.\ Kadison and V.\ Ostrik for a correspondence.
JF is partially supported by VR under project no.\ 621-2006-3343.

 \newpage

 \newcommand\wb{\,\linebreak[0]} \def\wB {$\,$\wb}
 \newcommand\Bi[2]    {\bibitem[#2]{#1}}
\newcommand\BOOK[4]  {{\sl #1\/} ({#2}, {#3} {#4})}
\newcommand\inBO[9]  {{\em #8}, in:\ {\sl #1}, {#2}\ ({#3}, {#4} {#5}),
                     p.\ {#6--#7} {{\tt [#9]}}}
\newcommand\inBo[8]  {{\em #8}, in:\ {\sl #1}, {#2}\ ({#3}, {#4} {#5}), p.\ {#6--#7}}
\newcommand\JO[6]    {{\em #6}, {#1} {#2} ({#3}), {#4--#5}}
\newcommand\J[7]     {{\em #7}, {#1} {#2} ({#3}), {#4--#5} {{\tt [#6]}}}
\newcommand\Pret[2]  {{\em #2}, pre\-print {\tt #1}}
 \def\jf    {J.\ Fuchs}
 \def\adma  {Adv.\wb Math.}
 \def\anma  {Ann.\wb Math.}        
 \def\apcs  {Applied\wB Cate\-go\-rical\wB Struc\-tures} 
 \def\coma  {Con\-temp.\wb Math.}
 \def\comp  {Com\-mun.\wb Math.\wb Phys.}
 \def\fiic  {Fields\wB Institute\wB Commun.}
 \def\jktr  {J.\wB Knot\wB Theory\wB and\wB its\wB Ramif.}
 \def\jnmp  {J.\wB Non-li\-ne\-ar\wB Math.\wb Phys.} 
 \def\joal  {J.\wB Al\-ge\-bra}
 \def\jomp  {J.\wb Math.\wb Phys.}
 \def\jpaa  {J.\wB Pure\wB Appl.\wb Alg.}
 \def\jram  {J.\wB rei\-ne\wB an\-gew.\wb Math.}
 \def\kthe  {K-The\-o\-ry}
 \def\marl  {Math.\wb Res.\wb Lett.}
 \def\nupb  {Nucl.\wb Phys.\ B}
 \def\slnm  {Sprin\-ger\wB Lecture\wB Notes\wB in\wB Mathematics}
 \def\taac  {Theo\-ry\wB and\wB Appl.\wB Cat.}
 \def\topo  {Topology}
 \def\AMS    {{American Mathematical Society}}
 \def\CUP    {{Cambridge University Press}}
 \def\EMS    {{European Mathematical Society}}
 \def\SV     {{Sprin\-ger Ver\-lag}}
 \def\Ca     {{Cambridge}}
 \def\NY     {{New York}}
 \def\PR     {{Providence}}

\end{document}